\documentclass[11pt,a4paper]{amsart}
\usepackage{pdfsync}
\usepackage{mathptmx} % rm & math
\usepackage[scaled=0.90]{helvet} % ss
\usepackage{courier} % tt
\normalfont
\usepackage[T1]{fontenc}
\usepackage{color}
\usepackage{amssymb}
\usepackage{MnSymbol}
\usepackage{graphicx}
\usepackage{epstopdf}
\usepackage{thm-restate} %for restatable environment
\usepackage{mathtools}

\usepackage[shortlabels]{enumitem}

\usepackage{hyperref}

\renewcommand{\epsilon}{\varepsilon}
\renewcommand{\setminus}{\smallsetminus}
\renewcommand{\emptyset}{\varnothing}

\newtheorem{theorem}{Theorem}[section]

\newtheorem{proposition}[theorem]{Proposition}
\newtheorem{corollary}[theorem]{Corollary}
\newtheorem{lemma}[theorem]{Lemma}

\theoremstyle{definition}
\newtheorem{example}[theorem]{Example}

\newtheorem{definition}[theorem]{Definition}
\newtheorem{notation}[theorem]{Notation}
\newtheorem*{notation*}{Notation}

\newtheorem*{ack}{Acknowledgements}

\theoremstyle{remark}
\newtheorem{remark}[theorem]{Remark}
\newtheorem*{remark*}{Remark}
\newtheorem{fact}[theorem]{Fact}

\newcommand{\cd}{\operatorname{cd}}
\newcommand{\hd}{\operatorname{hd}}
\newcommand{\normal}{\lhd}
\newcommand{\supp}{\operatorname{supp}}
\newcommand{\Sym}{\operatorname{Sym}}
\newcommand{\FSym}{\operatorname{FSym}}
\newcommand{\Alt}{\operatorname{Alt}}
\newcommand{\fix}{\operatorname{Fix}}

\newcommand{\Q}{\mathbb Q}
\newcommand{\Z}{\mathbb Z}
\newcommand{\N}{\mathbb N}
\newcommand{\R}{\mathbb R}

\newcommand{\Ra}{\mathcal{R}}

\newcommand{\B}{\mathcal B}

\newcommand{\base}{\Phi}

\newcommand{\quotient}[2]{{\left.\raisebox{.2em}{$#1$}\middle/\raisebox{-.2em}{$#2$}\right.}}

\newcommand{\fpinfty}{{\FP}_{\infty}}

\newcommand{\FFF}{\operatorname{F}}
\newcommand{\FP}{\operatorname{FP}}

\newcommand{\Aut}{\operatorname{Aut}}
\renewcommand{\ker}{\operatorname{Ker}}
\newcommand{\im}{\operatorname{Im}}

\renewcommand{\implies}{\Rightarrow}

\newcommand{\Hom}{\operatorname{Hom}}

\newcommand{\lex}{\buildrel{\textrm{lex}}\over\prec}
\newcommand{\gfin}{G_{\textit{fin}}}
\newcommand{\kfin}{K_{\textit{fin}}}
\newcommand{\hfin}{H_{\textit{fin}}}
\newcommand{\lfin}{L_{\textit{fin}}}
\newcommand{\gamfin}{\Gamma_{\textit{fin}}}
\newcommand{\stab}{\operatorname{stab}}
\renewcommand{\wr}{\operatorname{\,wr\,}}
\newcommand{\Wr}{\operatorname{\,\overline{wr}\,}}

\title[Subgroups of Houghton Groups]
{Finiteness properties of Subgroups of Houghton Groups of full Hirsch length}
\author{Charles Garnet Cox, Peter H. Kropholler, and Armando Martino}
\address{School of Mathematics, University of Bristol, Bristol BS8 1UG}
\email{charles.cox@bristol.ac.uk}
\address{Mathematical Sciences, University of Southampton, Southampton SO17 1BJ}
\email{p.h.kropholler@soton.ac.uk}
\email{a.martino@soton.ac.uk}
\date{\today} % Activate to display a given date or no date



\begin{document}

\begin{abstract}
In the 1980’s K.S.\ Brown proved that the Houghton group $H_n$ is of type $\FFF_{n-1}$ but not $\FP_n$. We show that, provided $n\ge3$, the same conclusion holds for all subgroups $G$ of $H_n$ that are \emph{large} in the sense that there is an epimorphism $G\twoheadrightarrow\Z^{n-1}$.

Our research leads naturally to the study of generalised permutational wreath products in which the base of the wreath product is a direct product of finite groups which are allowed to vary in isomorphism type from one orbit to another. Such generalised wreath products arise naturally amongst the large subgroups of Houghton groups and are accommodated by a generalised Jordan--Wielandt theorem.
\end{abstract}

\maketitle

\tableofcontents

\section{Introduction}
Let $H_n$ denote Houghton's group on $n$ rays. 
Brown proved \cite{Brown1987} that for each $n\ge1$, $H_n$ is of type $\FFF_{n-1}$ but not $\FP_n$. The main goal of this paper is to show that if $n$ is at least $3$ then \emph{large} subgroups of Houghton's group $H_n$ satisfy the same finiteness conditions as $H_n$ itself. By \emph{large} we mean \emph{having the same Hirsch length or rational homological dimension}. Our results employ a rich tapestry of ideas, including new results about permutational wreath products.
We refer the reader to \cite{BB,Bieri,Brown1987} for background information about the geometric and cohomological finiteness conditions, type $\FFF_n$ and type $\FP_n$. 

\begin{notation*} The set of natural numbers including $0$ is denoted by $\N$. (This convention differs from that chosen by Brown \cite[\S5]{Brown}.) 
\end{notation*}
The groups $H_1, H_2, H_3, \ldots$ constitute a family introduced in \cite{MR521478} by Houghton. Given $n\in \{1, 2, 3, \ldots\}$, 
the ray system $\Ra_n$ is defined formally to be $\{1,\ldots,n\}\times \N$: informally $\Ra_n$ is a disjoint union of $n$ \emph{rays} each of which is a copy of the set $\N$. Houghton's group $H_n$ consists of all permutations $g$ of $\Ra_n$ for which there is an integer vector $(t_1(g),\ldots, t_n(g))\in\Z^n$ with $(j,\ell)g=(j,\ell+t_j(g))$ for all $j$ and all sufficiently large $\ell$. For such a $g$ this integer vector is uniquely determined by $g$ and it is necessarily the case that $t_1(g)+\cdots+t_n(g)=0$. Furthermore, the map $g\mapsto (t_1(g),\ldots, t_n(g))$ defines a homomorphism $t$ from $H_n$ to $\Z^n$ whose image is in the set of zero-sum vectors and whose kernel is the set of finitary permutations of $\Ra_n$, which we denote by $\FSym(\Ra_n)$. It is easy to see that any zero-sum vector in $\Z^n$ arises as the \emph{translation vector} of a suitable permutation $g$ and therefore there is an exact sequence of groups
$$1\to \FSym(\Ra_n)\to H_n\buildrel{t}\over\to \Z^n\buildrel\epsilon\over\to\Z\to 0$$ in which $\epsilon$ is the \emph{augmentation map} given by 
$(m_1,\ldots,m_n)\buildrel\epsilon\over\mapsto m_1+\cdots+m_n$.
Thus $H_n$ is an extension of the countably infinite finitary symmetric group by a free abelian group of rank $n-1$. In particular $H_n$ is countable and \emph{elementary amenable}, the latter meaning that it belongs to the smallest class of groups closed under directed unions and extensions that contains all finite and all abelian groups.

For any group $G$ we write $\hd_\Q(G)$ and $\cd_\Q(G)$ for the homological and cohomological dimensions of $G$ over the field $\Q$ of rational numbers. For an elementary amenable group $G$ we write $h(G)$ for its Hirsch length. 

  In general, the Hirsch length $h(G)$ is defined for any elementary amenable group $G$ and takes its value in the set $\N\cup\{\infty\}$. An elementary amenable group has finite Hirsch length if and only if it admits a subnormal series in which the factors are infinite cyclic or locally finite and in that case the Hirsch length is precisely the number of infinite cyclic factors.

\begin{notation*} For any subgroup $G$ of $H_n$ we write $\gfin$ for the subset $G\cap\FSym(\Ra_n)$ of finitary permutations. We say that a subgroup $G$ of $H_n$ \emph{has full Hirsch length} when $G$ has the same Hirsch length as $H_n$.
\end{notation*}

In particular $H_n$ is an elementary amenable group of Hirsch length $n-1$. For an account of elementary amenable groups and the concept of Hirsch length, see \cite{HL} and \cite[Definition I.15]{Bridson2015}.

We shall also describe $H_n$ as the group of almost order preserving bijections of $\Ra_n$, when the latter is given the lexicographic ordering - see Proposition~\ref{aopisHou}. 

Our first significant result follows directly from a result of Leary--Nucinkis, \cite{Leary2001} together with an analysis of abelian subgroups of Houghton groups due to St. John-Green.

\begin{restatable*}{theorem}{NoFPn}\label{NoFPn}
	For any $n\ge1$ and any subgroup $G$ of $H_n$ the following are equivalent.
	\begin{enumerate}
	\item $G$ is of type $\FP_n$
	\item $G$ is of type $\FP_k$ for some $k\ge n$
	\item $G$ is of type $\fpinfty$
	\item $\gfin$ is finite and $G/\gfin$ is free abelian of rank at most $\lfloor\frac n2\rfloor$
	\item $\gfin$ is finite.
	\end{enumerate}
Moreover if $G$ satisfies these conditions then either 
$h(G)<h(H_n)$ or $n<3$. 
\end{restatable*} 
This is easily deduced by combining work of Hillman--Linnell and Stammbach with more recent work of Leary--Nucinkis, and is proved in the next section.

The main theorem of this paper asserts that subgroups of Houghton's groups of full Hirsch length typically satisfy the same finiteness conditions as their ambient Houghton group. Recall that \textbf{max-n} is the ascending chain condition on normal subgroups -- see Definition~\ref{maxn}. 

\begin{restatable*}{theorem}{main}	\label{main}
Fix $n\ge 2$. Let $G$ be a subgroup of $H_n$ that has full Hirsch length. Then $G$ is of type $\FFF_{n-1}$ and has \textbf{max-n}. Moreover, 
	\begin{enumerate}
		\item If $n \ge 3$ then $G$ is not of type $\FP_n$, 
		\item If $n=2$, then either $G$ is not of type $\FP_2$ or $G$ is finite-by-$\Z$ (and so is of type $\fpinfty$).
	\end{enumerate}
\end{restatable*}

The reader may notice that we have nothing to say about the case $n=1$. Our theorem describes a uniform pattern of results for $n\ge3$.
For $n=2$, this pattern is broken.
The Hirsch length of $H_2$ is 1, and any subgroup $G$ of $H_2$ has Hirsch length 0 or 1. In summary, we have the following.
\begin{lemma} For $n=2$ and for any subgroup $G$ of $H_2$ we have
\begin{enumerate}
\item $G$ has Hirsch length 0 if and only if $G=\gfin$, and
\item $G$ has full Hirsch length if and only if $G$ contains an element of infinite order.
\end{enumerate}
\end{lemma}
Combination of this Lemma and our main Theorem stated above tells the complete story for finiteness conditions of subgroups of $H_2$.

When $n=1$ the situation is simple and there is really nothing to say. Houghton's first group $H_1$ is locally finite and its subgroups, which all have full Hirsch length in our sense, are either finite and so $\fpinfty$ or infinite and so not $\FP_1$. In this case there is nothing useful one can say about the \textbf{max-n} condition which holds for some but not all (infinite) subgroups.

In order to prove Theorem~\ref{main}, we prove the following structure theorem. We define the notion of a strongly orbit primitive action in Definition~\ref{def:stronglyorbitprimitive} and multi-wreath products in Section \ref{sec:multiwreath}.\\

\begin{restatable*}{theorem}{structure}  \label{thm:structure}
Let $n\in \{3, 4, \ldots\}$ and $G$ be a subgroup of $H_n$ with $h(G)=n-1$. Then $G$ is abstractly commensurable to $\mathcal{W} \wr \Gamma$, a restricted multi-wreath product, where:
\begin{enumerate}
	\item $\mathcal{W}=\{W_1, \ldots, W_k\}$ and $W_1, \ldots, W_k$ are finite groups;
	\item $\Gamma$ is a subgroup of full Hirsch length of the $n$th Houghton group; and
	\item $\Gamma$ acts on the ray system strongly orbit primitively and with only infinite orbits $\Omega_1, \ldots, \Omega_k$.
\end{enumerate}
\end{restatable*}

\begin{remark*}
	Although we do not use this, it is fairly straightforward to show that any group one can construct as above can be embedded into $H_n$ as a subgroup of full Hirsch length. 
\end{remark*}

Recall that two groups $A$ and $B$ are said to be \emph{abstractly commensurable} if there exists a group $C$ and monomorphisms $\phi_A: C\to A$ and $\phi_B: C\to B$ whose images are finite index in $A$ and $B$ respectively.

One example to bear in mind for Theorem \ref{thm:structure} is the group $H_n$ and $G\le H_n$ a point stabiliser of the ray system $\Ra_n$. Then $G$ is infinite index as a subgroup, but is abstractly isomorphic, and hence abstractly commensurable, to $H_n$. This is a subgroup of full Hirsch length - see Remark~\ref{point stab}.

\medspace

\noindent\textit{Organisation}. Most of our work is in proving Theorem \ref{thm:structure}, with Section \ref{sec:thmAproof} using this to prove Theorem \ref{main}. We begin with $G\le H_n$ of full Hirsch length. In Section \ref{sec:fullhirsch} we show that up to abstract commensurability we may assume that $\pi(G)=(d\Z)^{n-1}$ for some $d\in\N$ and $G$ has orbits $\mathcal{O}_1, \ldots, \mathcal{O}_k$, each infinite. In Section \ref{sec:multiwreath} we show, under certain hypotheses (which are verified in Section \ref{sec:cameron}) that such a group is finite index in a multi-wreath product. In Section \ref{sec:subdirect} we consider subdirect products. We do so since the head of our wreath product is subdirect in $\Gamma_1\times\cdots\times\Gamma_k$ where $\Gamma_i$ is finite index in $H_n$ for $i=1, \ldots, k$ and the head contains the full product of the alternating groups for each factor. We then use results from \cite{KM1} to deduce the finiteness properties of our multi-wreath products via the BNS invariants of certain subdirect products of Houghton groups. 

An important ingredient in our arguments is a version of the Jordan--Wielandt theorem for infinite permutation groups. Our version of the Jordan--Wielandt theorem allows non-transitive actions with a finite number of infinite orbits and may be stated as follows. 

\begin{restatable*}{theorem}{JordanWielandt}\label{prodalt} Let $\Omega$ be an infinite set and $\Gamma\le\Sym(\Omega)$ satisfy:
\begin{enumerate}[(i)]
\item The orbits of $\Gamma$ are exactly $\Omega_1, \ldots, \Omega_k$, with each of these being infinite and $\bigcup_{i\le k}\Omega_i=\Omega$. Hence $\Gamma\le \Sym(\Omega_1)\times\cdots\times\Sym(\Omega_k)$, a subgroup of $\Sym(\Omega)$.
\item\label{two}  There exists $\gamma\in \Gamma$, where $\gamma$ is finitary and has support that meets every orbit $\Omega_1, \ldots, \Omega_k$.
\item\label{three} The action of $\Gamma$ on $\Omega$ is strongly orbit primitive.
\end{enumerate}
Then $\Gamma\ge \bigoplus_{i=1}^k\Alt(\Omega_i)$.
\end{restatable*}

\begin{ack}
	We dedicate this paper to the memory of Chris Houghton, who died on 6 May 2024, while this work was in preparation. 
\end{ack}

\section{Preliminaries}

\begin{notation} For a set $X\ne \emptyset$, let:
\begin{itemize}
\item $\Sym(X)$ denote the set of all permutations of $X$;
\item $\supp(g):=\{x\in X : xg\ne x\}$ for any $g\in\Sym(X)$;
\item $\FSym(X):=\{g\in \Sym(X) : |\supp(g)|<\infty\}$; and
\item $\Alt(X)\le \FSym(X)$ consists of all even permutations.
\end{itemize}
If it is clear from the context, then we will omit the set $X$ from these notations.
\end{notation}
\begin{lemma}\label{lem:AltisSimple}
If $X$ is an infinite set, then $\Alt(X)$ is simple.
\end{lemma}
\begin{proof} Let $N$ be a non-trivial normal subgroup of $\Alt(X)$, and take $\sigma\in N\setminus\{1\}$. Then $\supp(\sigma)$ is finite, and so contained within a finite $F\subset X$ where without loss of generality we can assume that $|F|\ge 5$. Let $A:=\Alt(F)$. Then $N\cap A$ is a normal non-trivial subgroup of $A$ and hence $N\cap A=A$. Then $N$ contains a 3-cycle and so, by normality, every 3-cycle in $\Alt(X)$. Since $\Alt(X)$ is generated by the set of all 3-cycles, we conclude the $N=\Alt(X)$.
\end{proof}
We recall the definition of the Houghton groups, $H_n$, for reference.
\begin{definition} Let $n\in \{2, 3, \ldots\}$. We define $H_n\le \Sym(\Ra_n)$ to consist of those $g\in \Sym(\Ra_n)$ for which there is an integer vector $(t_1(g),\ldots, t_n(g))\in\Z^n$ and $n_g\in \N$ where $(j,\ell)g=(j,\ell+t_j(g))$ for all $j$ and all $\ell\ge n_g$. Observe that $H_1=\FSym(\Ra_1)$.
\end{definition}

From our above definition, it is perhaps not obvious that the groups $H_2, H_3, \ldots$ are finitely generated. We begin by describing, for each of these groups, a finite generating set. These can be found, for example, in \cite{MR3078485}.

\begin{notation}\label{not:permsgj}
Let $n\in\{2, 3, \ldots\}$ and $j\in\{2, \ldots, n\}$. Then $g_j$ is the permutation of $\Ra_n$ whose support is equal to $R_1\cup R_j$ and
\begin{equation*}
(i,m)g_j=\left\{\begin{array}{ll}(1, m+1) & \mathrm{if}\ i=1\ \mathrm{and}\ m\in\N\\ (1,0) & \mathrm{if}\ i=j\ \mathrm{and}\ m=0\\ (j, m-1) & \mathrm{if}\ i=j\ \mathrm{and}\ m\in\{1, 2, \ldots\}.
\end{array}\right.
\end{equation*}
\end{notation}

\begin{remark*} For $n\ge3$ we have that $H_n=\langle g_2, \ldots, g_n\rangle$. Also $H_2= \langle ((1, 1)\;(1,2)), g_2\rangle$.
\end{remark*}

The following will be useful.

\begin{lemma}[\cite{MR3302573}, Corollary 2.8]
	\label{3.1}
	Every element of $H_n$, considered as a permutation group on $\Ra_n$, is uniquely expressible as a product of finitely many disjoint cycles. For $g\in H_n$, the number of infinite cycles in its makeup is $\frac{1}{2}\sum_{j=1}^n|t_j(g)|.$
\end{lemma}
The following will be a profitable reformalisation of a Houghton group. Consider a lexicographical ordering $\lex$, defined by $$(j,k)\lex(j',k')$$ if and only if either $j<j'$ or $j=j'$ and $k<k'$. Under the lex order, $\Ra_n$ is well-ordered with order type $\omega n$; see Figure \ref{fig}. We can define the Houghton groups in terms of this ordering.

\begin{figure}[htbp] %  figure placement: here, top, bottom, or page
	\centering
	\includegraphics[width=3in]{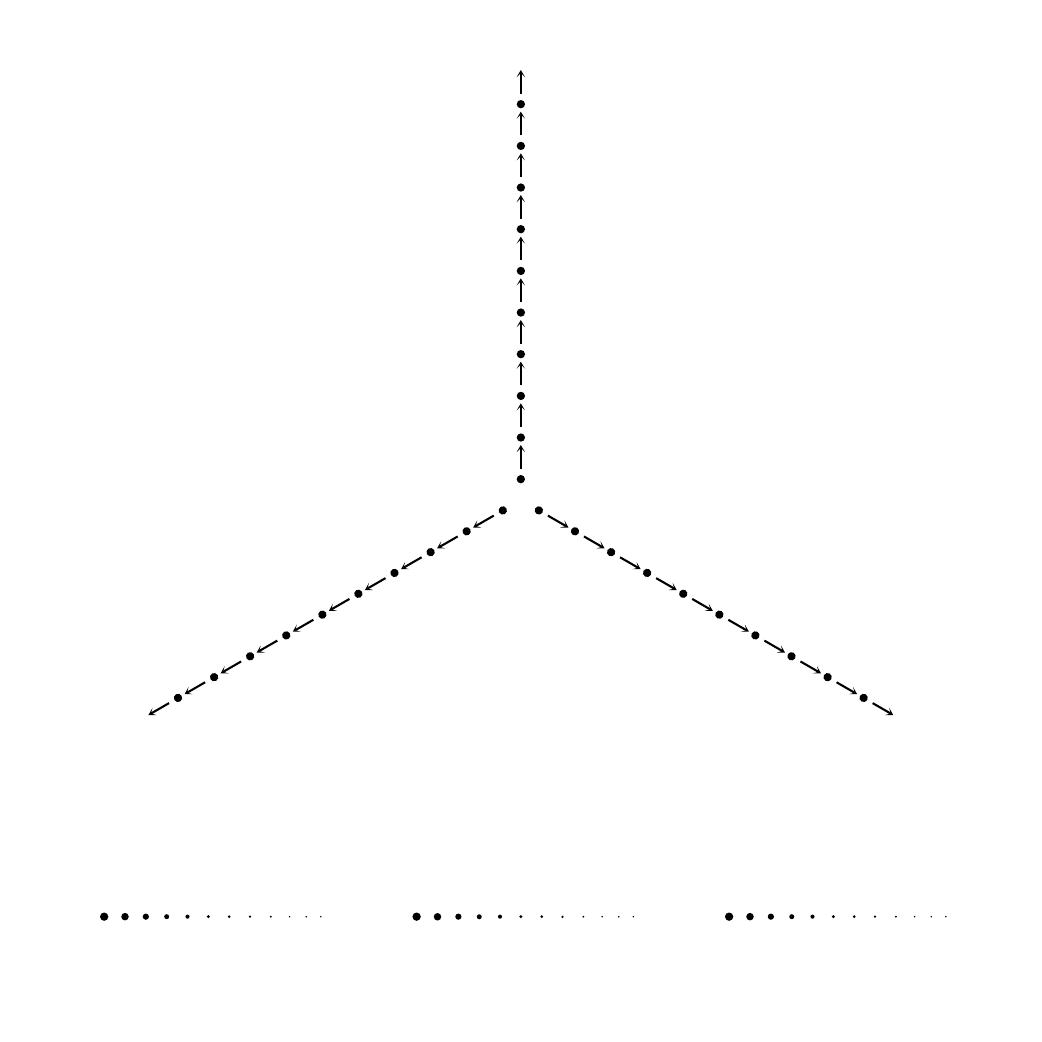}
	\vspace{-0.75cm}
	\caption{The ray set $\Ra_3$ ordered as three rays above and lexicographically below.}
	\label{fig}
\end{figure}

\begin{definition}
	Let $(S,\prec)$ and $(T,\prec)$ be posets. Then a function $f:S\to T$ is \emph{almost order preserving} provided there is a finite set $F\subseteq S$ such that for any pair $s\prec s'$ in $S\setminus F$ we have $f(s)\prec f(s')$ in $T$.
\end{definition}
The following well-known results about orderings will be useful.
\begin{fact}\label{orders1} If $S$ and $T$ are well-ordered, then there exists a partial order isomorphism $f: S\to T$ which sends the initial segments of $S$ to the initial segments of $T$ and either the domain of $f$ is $S$ or the image of $f$ is $T$.
\end{fact}
\begin{fact}\label{orders2} If $S$ is well-ordered, then each $x\in S$ is exactly one of the following:
\begin{itemize}
\item[(i)] the `0' element, which is $\min S$;
\item[(ii)] a successor element, if $\{s\in S : s<x\}$ has a maximum element; or
\item[(iii)] a limit point, if $x\ne0$ and $\{s\in S : s<x\}$ has no maximum element.
\end{itemize}
\end{fact}
Note that $\Ra_n$ with the lexicographic order has `0' element $(1, 0)$ and exactly $n-1$ limit points (which are the points $(2, 0), \ldots, (n, 0)$).
\begin{lemma}\label{lem:orderiso1}
	Let $n\in \{1, 2, \ldots\}$, $\Ra_n$ be the ray system and $F$ be a finite subset of $\Ra_n$. Then $\Ra_n-F$ is order isomorphic to $\Ra_n$. 
\end{lemma}
\begin{proof}
Let $i\in \{1,\ldots, n\}$. Then there is an order isomorphism from $R_i$, the $i$th ray of $\Ra_n$, to $R_i-F$. Combine these for an order isomorphism between $\Ra_n$ and $\Ra_n-F$.
\end{proof}
We now focus on permutations, rather than just sets.
\begin{lemma}\label{H1orderiso} Let $f: \N\setminus F_1\to \N\setminus F_2$ be an order isomorphism. Then there exists $n_f\in \N$ and $t(f)\in \Z$ such that $f(n)=n+t(f)$ for all $n\ge n_f$.
\end{lemma}
\begin{proof} Let $l=\max\{\sup(F_1)+1, f^{-1}(\sup(F_2)+1)\}$. Fix $k\in \N$. Then the successor of $l+k$ in $\N\setminus F_1$ is $l+k+1$ and the successor of $f(l)+k$ in $\N\setminus F_2$ is $f(l)+k+1$. Thus, for any $r\in \N$, we have that $f(l+r)=f(l)+r$.
\end{proof}
The Houghton groups can then be seen as the group of almost order preserving bijections.
\begin{proposition}\label{aopisHou}
Let $n\in \{1, 2, \ldots\}$. Then the subgroup of $\Sym(\Ra_n)$ of almost order preserving elements (with respect to the lexicographical order) is $H_n$.
\end{proposition}
\begin{proof}
Throughout, when we say `order' we will mean the lexicographical order. If $h\in H_n$, then $h$ is almost order preserving. Now take $g\in \Sym(\Ra_n)$ that is almost order preserving on $\Ra_n$. Thus there are finite sets $F_1$ and $F_2$ such that $g$ is an order isomorphism from $\Ra_n\setminus F_1$ to $\Ra_n\setminus F_2$. If $n=1$, then $|F_1|=|F_2|$ and the quantity $t(g)$ in Lemma \ref{H1orderiso} must be zero, meaning that $g\in \FSym$. Now let $n\ge 2$. Assume $g(1, m)\not\in (\Ra_n\setminus F_2)\cap R_1$ for some $m\in \N$. Then $g(1, m)=(i, m')$ for some $i\in\{2, \ldots, n\}$. Since $g$ is order preserving, this means for all $k\in \N$ that $g(1, m+k)\not\in R_1$, contradicting that $g$ is surjective. Thus $g$ induces an order isomorphism from $(\Ra_n\setminus F_1)\cap R_1$ to $(\Ra_n\setminus F_2)\cap R_1$. Running the same argument for each $j=2, \ldots, n$ shows that $g$ induces an order isomorphism from $(\Ra_n\setminus F_1)\cap R_j$ to $(\Ra_n\setminus F_2)\cap R_j$. Lemma \ref{H1orderiso} then yields the result.
\end{proof}

\begin{remark} \label{point stab}
Note that Lemma~\ref{lem:orderiso1} and Proposition~\ref{aopisHou} imply that a point stabiliser in $H_n$ is an infinite index subgroup isomorphic to $H_n$. Since a point stabiliser is isomorphic to $H_n$, it has the same Hirsch length and is therefore an infinite index subgroup of full Hirsch length. 	
\end{remark}

\section{Elementary amenable groups and Rational Homological Dimension}

The goal of this section is to establish Theorem \ref{NoFPn}. 

The notion of rational homological dimension applies to any group. If $G$ is a group then its rational homological dimension is denoted by $\hd_\Q(G)$. 
Stammbach \cite{Stammbach} refers to \emph{weak (homological) dimension} of the group algebra $KG$ where $K$ is any field of characteristic zero and denotes this by $\mathrm{w}.\dim(KG)$. It is easy to show that for any field $K$, $\mathrm{w}.\dim(KG) = \hd_K(G)$. Moreover, for any field $K$ of characteristic zero $\hd_K(G)=\hd_\Q(G)$. 
In \cite{Stammbach}, Stammbach fixes on the class of groups admitting a subnormal series in which the factors are locally finite or abelian. 
In the language of homological dimension, the main result of Stammbach's paper \cite{Stammbach} says that for a group $G$ in this class we have $h(G)=\hd_\Q(G)$. This covers the cases of interest to us because Houghton's groups are (locally finite)-by-abelian and are covered by Stammbach's paper. Nevertheless, it is worth remarking that the notion of Hirsch length is extended by Hillman and Linnell to encompass all elementary amenable groups, and the following overall picture emerges.

\begin{theorem}[The Hillman--Linnell--Stammbach Theorem \cite{HL,Stammbach}]
Let $G$ be an elementary amenable group. Then both Hirsch length and rational homological dimension are defined and equal. 
\end{theorem}

The following lemma will be useful.

\begin{lemma}\label{finite-by-abelian_groups}
Let $G$ be a finitely generated group that has a finite normal subgroup $F$ such that $G/F$ is abelian. Then the centre of $G$ has finite index in $G$.
\end{lemma}
\begin{proof}
For readers' convenience we include details. 
Since $G/F$ is abelian it follows that every commutator $[g,h]$ belongs to $F$ and hence there are only finitely many commutators. Now for $g\in G$, we have $g^h=g[g,h]$ for any $h$ and hence $g$ has only finitely many conjugates. Hence $C_G(g)$ has finite index in $G$, the index of the centraliser being equal to the size of the conjugacy class. Now if $G$ is generated by $g_1,\dots,g_d$ then the intersection of the centralisers 
$\bigcap_{i=1}^dC_G(g_i)$ is equal to the centre of $G$ and has finite index.
\end{proof}

We shall also need the following important result of Leary and Nucinkis \cite{Leary2001}.

\begin{theorem}[Leary--Nucinkis]\label{L--N}
Let $G$ be a group of type $\FP_n$ where $n>\hd_\Q(G)$. Then there is a bound on the orders of the finite subgroups of $G$.
\end{theorem}

\NoFPn

\begin{proof}[Proof of Theorem \ref{NoFPn}]
The implications (iv)$\implies$(iii)$\implies$(ii)$\implies$(i) are clear. 
If (v) holds then Lemma \ref{finite-by-abelian_groups} shows that $G$ is abelian-by-finite and St. John-Green's analysis \cite[Theorem 3.4]{Simon} confirms (iv).
Now assume that (i) holds. Then the hypotheses of Theorem \ref{L--N} are satisfied and so there is a bound on the orders of the finite subgroups of $G$. Since $\gfin$ is locally finite, it is finite, and (v) holds. 

When the conditions (i)--(v) all hold we have $h(G)=\lfloor\frac{n}{2}\rfloor<n$ since $n\ge1$. In fact $h(G)=\lfloor\frac{n}{2}\rfloor<n-1=h(H_n)$ whenever $n>2$, confirming the last sentence of
Theorem \ref{NoFPn}.
\end{proof}

\section{Subgroups of full Hirsch length}\label{sec:fullhirsch}

Recall that a group has Hirsch length $m$ if it admits a subnormal series whose factors are locally finite or cyclic and exactly $m$ of those factors are infinite cyclic. 

Also, given a subgroup $G \leq H_n$, recall that $\gfin$ denotes $G \cap \FSym(\Ra_n)$. It is straightforward to see that  $\gfin\normal G$, since it is the intersection of $G$ with a normal subgroup of $H_n$. In fact it is characteristic in $G$, since it consists of all the finite order elements of $G$. 

Moreover, by the Second Isomorphism Theorem, 
$$G/\gfin \cong G \FSym(\Ra_n) / \FSym(\Ra_n) \leq H_n / \FSym(\Ra_n) \cong \Z^{n-1}.$$ 

It follows that the Hirsch length of $G$ is $n-1$ precisely when $G/\gfin$ is isomorphic to $\Z^{n-1}$. In fact, this is equivalent to saying that the image of $G$ in the abelianisation of $H_n$ has finite index, since $\FSym(\Ra_n)$ is the derived subgroup of $H_n$ for $n \geq 3$ and $\Alt(\Ra_2)$ is the derived subgroup of $H_2$. We note this observation for later use.

\begin{lemma}\label{lem:structureofkfin}
Let $n\in \{2, 3,\ldots\}$, and $G \le H_n$. Then the following are equivalent,
\begin{itemize}
	\item $h(G)=n-1$, 
	\item $G/\gfin\cong\Z^{n-1}$, 
	\item The image of $G$ in the abelianisation of $H_n$ is a finite index subgroup. 
\end{itemize}
\end{lemma}

\begin{remark} \label{full point stabs} 
	In particular, the second condition is useful since $\gfin$ consists of all the torsion elements of $G$. 
\end{remark}

Next we move on to discussing the action of a subgroup of $H_n$ on the ray system, $\Ra_n$. 

\begin{lemma}
	\label{inforbs}
Let $n\in \{2, 3,\ldots\}$ and $G\le H_n$ have full Hirsch length, and consider the action of $G$ on the ray system $\Ra_n$. Then every infinite orbit of $G$ meets each ray in an infinite set.
\end{lemma}
\begin{proof}
Let $v \in \Ra_n$ and suppose that $vG$ is infinite. Then it must meet some ray, say ray $i$, in an infinite set. Now consider another arbitrary ray, $R_j$. Since $G$ has full Hirsch length, there exists a $g \in G$ such that $t_i(g) < 0, t_j(g) > 0$ and $t_k(g) = 0$ for all $k \neq i,j$. Lemma \ref{3.1} states that $g$ can be written as a product of finitely many disjoint cycles. Thus, by construction, the support of $g$ is almost equal to $R_i\cup R_j$. In particular, $vG\cap R_i$ must contain an element of the support of $g$ that is not part of a finite cycle (and so is not in a finite orbit of $g$). Call this element $u \in \Ra_n$. By choosing a sufficiently large $z\in\N$, we will have that $\{u g^m : m\ge z\}$ is a subset of $R_j$. Hence $vG$ meets $R_j$ in an infinite set.
\end{proof}

\begin{lemma}\label{4.4}
Let $n\in \{2, 3,\ldots\}$ and $G\le H_n$ have full Hirsch length. Then $G$ acts on $\Ra_n$ with finitely many orbits.
\end{lemma}
\begin{proof}
Fix an arbitrary ray $R_j$. Then, using that $h(G)=n-1$, there exists a $g\in G$ with $t_j(g)=d>0$. For almost all $(j, m) \in R_j$, we have that $(j, m)g=(j, m+d)$. Thus $R_j$ intersects finitely many orbits, and so does $\Ra_n$.
\end{proof}

The following will be helpful to us later.
\begin{lemma}\label{lem:finitaryonallorbits} Let $n\in \{3, 4, \ldots\}$ and $G\le H_n$ have full Hirsch length. Then there exists $\sigma\in \gfin$ whose support has non-trivial intersection with each infinite orbit of $G$.
\end{lemma}
\begin{proof}
	From the full Hirsch length assumption, there exist elements $h_2, h_3\in G$ with 
	$$
	\begin{array}{rcl}
		t_1(h_2) <0,  & & t_1(h_3) < 0, \\
		t_2(h_2) > 0, & & t_3(h_3) > 0, \\
		t_i(h_2) =0, i \neq 1,2  & & 	t_i(h_3) =0, i \neq 1,3 \\
	\end{array}
	$$ 
	
That is, $h_2$ translates along rays 1 and 2 (sufficiently far out) and fixes all but finitely many points of every other ray. Similarly, for $h_3$ with respect to rays 1 and 3. For $i=1, 2$ we may  replace $h_i$ with a larger power so to assume that each $x\in \supp(h_i)$ lies on an infinite orbit of $\langle h_i\rangle$, and we assume that we have done this.

We claim that the commutator $h_2h_3h_2^{-1}h_3^{-1}$ is a suitable choice for our $\sigma\in \gfin$. (Note this element is in the commutator subgroup and so automatically lies in $\gfin$).

	To proceed we consider some infinite $G$-orbit, $\mathcal{O}$. 
	By Lemma~\ref{inforbs}, $\mathcal{O}$ meets every ray, and in particular the first ray, in an infinite set. Since the supports of both $h_2$ and $h_3$ meet the first ray in a co-finite set, we have that $\mathcal{O} \cap \supp(h_2) \cap \supp(h_3) \neq \emptyset$. Therefore, we can choose some $y_0 \in \mathcal{O} \cap \supp(h_2) \cap \supp(h_3)$ and we now define $y_j = y_0 h_3^j$, for $j \geq 0$. 
	
	Now, for sufficiently large $j$, $y_j$ lies on the third ray, since $y_0$ is in the support of $h_3$ and therefore lies in an infinite orbit of $\langle h_3 \rangle$ - and also recalling that $t_1(h_3)<0$ and $t_3(h_3) > 0$. In particular, for sufficiently large $j$, $y_j \not\in \supp(h_2)$. However, $y_j \in \mathcal{O} \cap \supp(h_3)$ for all $j$. Therefore there is a greatest integer $j$ such that $y_j \in \mathcal{O} \cap \supp(h_2) \cap \supp(h_3)$, whereas $y_{j+1} \in \mathcal{O} \cap \supp(h_3) \setminus \supp(h_2)$. Let $z = y_j$. Then 
	$$
	z h_3 h_2 {h_3}^{-1} = z \neq z h_2 \implies z h_3 h_2 {h_3}^{-1} {h_2}^{-1} \neq z. 
	$$
	
	Therefore $h_2h_3h_2^{-1}h_3^{-1}$ acts non-trivially on the orbit $\mathcal{O}$ for an arbitrary infinite orbit, as required. 
\end{proof}

We now show that, up to abstract commensurability, we can assume $G$ has only infinite orbits. Recall the standard notations that $G_p:=\stab_G(p)$ and $A\le_f B$ denotes that $A$ has finite index in $B$.
\begin{lemma} \label{commensurable} With $n\in \{2, 3, \ldots\}$ and $G\le H_n$ of full Hirsch length acting on the ray system, $\Ra_n$, we get the following:
\begin{enumerate}
\item The set $F$, obtained by taking the union of the finite orbits of $G$, is finite.
\item $K:=\stab_G(F)$ has finite index in $G$ and each of its orbits are either infinite or a singleton.
\item The $K$ above is isomorphic to some $L\le H_n$ so that the orbits of $L$ are all infinite and $L$ has full Hirsch length.
\end{enumerate}
 
In particular, $G$ is abstractly commensurable to a subgroup $L$ of $G$ of full Hirsch length and so that every $L$-orbit on $\Ra_n$ is infinite.  
\end{lemma}

\begin{proof} Part (i) follows from  Lemma~\ref{4.4}. For (ii), $K$ is the kernel of the restriction map from $G$ to $\Sym(F)$ and hence has finite index.

Take $p\in \Ra_n$ such that $pG$ is infinite. This is equivalent to $[G: G_p]=\infty$. Using that $K\le_f G$ and $K_p\le_fG_p$, we see that $[K: K_p]=\infty$ meaning that $pK$ is also infinite.

Finally we show (iii).  By Lemma \ref{lem:orderiso1} we have that $K$, considered as a subgroup of $\Sym(\Ra_n\setminus F)$ is isomorphic, using the induced order preserving bijection from $\Ra_n\setminus F$ to $\Ra_n$, to a group $L$ of almost order preserving bijections on $\Ra_n$. Hence, by Proposition \ref{aopisHou}, $K$ is isomorphic to a subgroup $L$ of $H_n$ where the orbits of $L$ are all infinite and partition $\Ra_n$. Since they are isomorphic, we have that
$$\Z^{n-1}\cong\quotient{K}{\kfin}\cong \quotient{L}{\lfin}.$$
By Lemma \ref{lem:structureofkfin}, $L$ has full Hirsch length.
\end{proof}
Now that we have imposed a condition on the orbits of our $G\le H_n$, we will impose (at the cost of replacing $G$ by a finite index subgroup) a further restriction relating to the possible translation lengths. We first deal with $n\geq 3$.

\begin{definition}
	\label{level}
	Let $n\in \{2, 3, 4,\ldots\}$. We shall say that $G\le H_n$ is a \emph{level} subgroup if one of the following holds:
\begin{enumerate}[(i)]
	\item $n \geq 3$ and for each $g\in G$ and each pair $i\ne j$ of elements of $\{1,\ldots,n\}$, there exists $y\in G$ such that $y$ and $g$ have the same translation component on the $j$th ray and $y$ has translation component zero on the $i$th ray or, 
	\item $n=2$ and  for every $p \in \Ra_2$, $G_p \gfin$ is either equal to $G$ or to $\gfin$.  
\end{enumerate}

	Let $\pi$ denote the natural surjective map from $H_n$ to $\Z^{n-1}$. If $\pi(G)=(m\Z)^{n-1}$ for some integer $m\ge 1$, then we shall say that $G$ is a \emph{congruence-lifting subgroup}. That is, $G$ maps to a congruence subgroup of the abelianisation. 
\end{definition}

\begin{remark}
	It is easy to see that for $n=2$, any subgroup of full Hirsch length is a congruence-lifting subgroup. It also transpires that for $n \geq 3$, if $G$ is any level subgroup then $G_p \gfin = G$, by Proposition~\ref{6.1}.  
\end{remark}

The following is straightforward.

\begin{lemma}
	\label{level for n>3}
	Let $n\in \{3, 4,\ldots\}$. Any congruence-lifting subgroup of $H_n$ is level. In particular, any subgroup of full Hirsch length of $H_n$ admits a finite index subgroup which is level. 
\end{lemma}
\begin{proof}
We adopt the notation from Definition \ref{level}. Suppose that $\pi(G)=(m\Z)^{n-1}$ where $m\ge 1$. Given any $g\in G$, we can find $y\in G$ with $t_j(y)=t_j(g)$, $t_{j'}(y)=-t_j(y)$ for some $j'\not\in\{i, j\}$ and $t_k(y)=0$ otherwise. Hence $G$ is level. Assuming that $G$ has full Hirsch length, the 2nd claim follows since a finite index subgroup of $\Z^{n-1}$ contains $(p\Z)^{n-1}$ for some $p\ge 1$.
\end{proof}

\begin{remark}
It is easy to find subgroups of full Hirsch length which fail to be level. For example, when $n=3$ choose the subgroup $G$ generated by the finitary permutations together with two elements of infinite order whose translation vectors are
	$$(1,2,-3)$$
	and $$(2,1,-3).$$
	Thus there is an element with translation component $1$ on the first ray, but it is easy to see that any element with translation component zero on the second ray must have translation length that is a multiple of $3$ on the first ray.
\end{remark}

It also transpires that any subgroup of full Hirsch length in $H_2$ admits a level subgroup of finite index. 

\begin{proposition}
	\label{level for n=2}
	Let $G$ be a subgroup of full Hirsch length in $H_2$. Then $G$ admits a finite index subgroup which is level.  
\end{proposition}
\begin{proof} To show that $G$ has a level subgroup of finite index we will show that there is a finite index subgroup $K$ of $G$ such that $\kfin = \gfin$ and for each $p \in \Ra_2$, $K_p \kfin$ equals either $K$ or $\kfin$.

	To do this, note that each $G_p\gfin$ is a normal subgroup of $G$ (since any subgroup of $G$ containing $\gfin$ is normal) and is either equal to $\gfin$ or is a finite index subgroup of $G$, since $G/\gfin \cong \Z$. In particular, since $(G_p \gfin)^g = G_{pg} \gfin$ is normal and hence equals $G_p \gfin$, there are only finitely many subgroups $G_p  \gfin$, as $p$ varies in $\Ra_2$ (at most one for each $G$-orbit). Therefore we can take $K$ to be the intersection of all $G_p \gfin$ which have finite index in $G$. This is therefore a finite intersection of finite index subgroups of $G$ and so has finite index. Clearly, $K$ contains $\gfin$ and hence $\kfin = \gfin$. We claim that $K$ is level.

	 To see this, first choose any $p \in \Ra_2$. Then $K_p \kfin = K \cap (G_p \gfin)$, since $K_p = K \cap G_p$ and $\kfin=\gfin \cap K = \gfin$. Hence if $G_p \gfin$ has finite index in $G$, then it contains $K$ and $K_p \kfin=K$. Similarly if $G_p \gfin = \gfin$, then  $K_p \kfin=\kfin$. 
\end{proof}

We gather these results together in the  following.

\begin{proposition}\label{level and inforbs}
	Let $n\in \{2, 3, \ldots\}$ and $G$ be a subgroup of $H_n$ of full Hirsch length. Then $G$ is abstractly commensurable to a subgroup, $K$, of $H_n$ which is level and whose orbits are all infinite. 
\end{proposition}
\begin{proof}
	We already know that $G$ is abstractly commensurable to a subgroup $L$ of $H_n$ whose orbits are all infinite, by Lemma~\ref{commensurable}. Now by Lemmas~\ref{level for n>3} and \ref{level for n=2}, $L$ admits a finite index subgroup $K$ which is level. While $L$ and $K$ need not have the same orbits, the fact that $L$ admits no finite orbits implies that neither does $K$. This is because the size the orbit of some $p \in \Ra_n$ is simply the index of the point stabiliser. But if $L_p$ has infinite index in $L$ then $K_p = K \cap L_p$ must have infinite index in $K$. 
\end{proof}

We also obtain the following.
\begin{lemma}
	\label{restricting}
	Let $n\in \{2, 3, \ldots\}$ and $G$ be a subgroup of $H_n$ of full Hirsch length. Then any infinite orbit $\mathcal{O}$ of $G$ is order isomorphic to $\Ra_n$. Moreover, 
	\begin{enumerate}[(i)]
		\item The restriction of $G$ to $\mathcal{O}$ is a subgroup of $H_n$ (of $\mathcal{O}$) of full Hirsch length, and
		\item If $G$ is level, then the restriction to an infinite orbit is also level. 
	\end{enumerate}
\end{lemma}
\begin{proof}
	Recall that by Proposition~\ref{aopisHou}, $H_n$ is precisely the group of almost order preserving bijections of $\Ra_n$, and so we may define the Houghton group on any set order isomorphic to $\Ra_n$.

	Next observe that Lemma~\ref{inforbs} says that $\mathcal{O}$ meets every ray in an infinite set, from which it follows that $\mathcal{O}$ is order isomorphic to $\Ra_n$ and hence the restriction map, $\textrm{res}_{\mathcal{O}}: G \to \Sym(\mathcal{O})$ lands in the group of almost order preserving bijections of $\mathcal{O}$. That is, the image of the restriction map lands in the Houghton group of $\Sym(\mathcal{O})$. 
	
	To prove (i), note that since $\mathcal{O}$ meets every ray in an infinite set, any element of the kernel of  $\textrm{res}_{\mathcal{O}}$ has zero translation vector. That is, $\ker(\textrm{res}_{\mathcal{O}}) \leq \gfin $ and therefore $G/\gfin \cong \Gamma/\gamfin$, where $\Gamma$ is the image of $\textrm{res}_{\mathcal{O}}$ and the isomorphism is induced by the restriction map. In particular, by Lemma~\ref{lem:structureofkfin}, $\Gamma$ has full Hirsch length. 
	
	To prove (ii), we split into two cases depending on $n$. For $n \geq 3$, and any $\gamma \in \Gamma = \im(\textrm{res}_{\mathcal{O}})$, we need to exhibit another element of $\Gamma$ whose translation vector agrees with $g$ on the $i^{th}$ ray and is zero on the $j^{th}$ ray. However, the rays of $\mathcal{O}$ are precisely the intersections of $\mathcal{O}$ with the original rays. In particular, since $G$ is level and there exists a $g \in G $ which maps to $\gamma$, we can find a $y \in G$ whose translation vector agrees with $g$ on ray $i$ and is zero on ray $j$. Our required element for $\Gamma$ is then the image of $y$. 
	
	For $n=2$, we need instead to prove that $\Gamma_p \gamfin$ equals either $\Gamma$ or $\gamfin$ for all $p \in \mathcal{O}$. However, it is easy to see that 
	$\textrm{res}_{\mathcal{O}}(\gfin) = \gamfin$, $\textrm{res}_{\mathcal{O}}(G_p) = \Gamma_p$ and $\textrm{res}_{\mathcal{O}}(G) = \Gamma$, from which it follows that, in the $n=2$ case, $G$ being level implies that $\Gamma$ is level. 
\end{proof}

\section{Generalising primitivity to intransitive actions}\label{sec:blocks}

\subsection*{Strongly Orbit Primitive}

The aim of this section is to introduce the concept of a strongly orbit primitive action, which will generalise the idea of a primitive group action to the intransitive case. This will be a condition which implies that the restriction to each orbit is primitive, but is stronger than that; it will also say that the orbits are also independent in some sense. 

We recall that a \emph{$G$-partition} (or a $G$-invariant partition) on a $G$-set $S$ is a partition $\mathcal{P}$ of $S$ such that $Pg \in \mathcal{P}$ for all $P \in \mathcal{P}$ and all $g \in G$. That is, $G$ permutes the parts of the partition. 

The action of a group $G$ on a set $S$ is called \emph{primitive} if, for any $G$-partition, either all the parts of the partition are singletons or there is a single part equal to $S$. Equivalently, the action is primitive if it is transitive and point stabilisers are maximal subgroups.

\begin{definition}[Strongly Orbit Primitive]
	\label{def:stronglyorbitprimitive}
	We say that a group $G$ acting on a set $S$ is \emph{strongly orbit primitive} if for every $G$-partition on $S$, every part of the partition is either a singleton or contains an entire $G$ orbit. 
\end{definition}

\begin{proposition}
		Let $G$ act on a set $S$. The following are equivalent: 
	\begin{enumerate}[(i)]
		\item The action of $G$ on $S$ is strongly orbit primitive. 
		\item For every $G$-partition on $S$, every part of the partition is either a singleton or a union of orbits.
		\item Every point stabiliser is a maximal subgroup of $G$ and points from different $G$ orbits have distinct stabilisers. 
	\end{enumerate}
\end{proposition}
\begin{proof}

	(i) $\Rightarrow$ (ii) is clear. Indeed, if $\mathcal{P}$ is a $G$-partition and $P \in \mathcal{P}$ is not a singleton, then $P$ contains a $G$ orbit, hence is stabilised by $G$ and is therefore a union of $G$ orbits. 
	
	(ii) $\Rightarrow$ (iii):  
	Fix $p\in S$ and let $H$ be a subgroup with $G_p \le H \le G$.  
	The partition whose parts are $pHg$ where $g \in G$ and singletons on other orbits 
	yields a $G$-invariant partition.  
	By (ii), the part containing $p$ is either a singleton (forcing $H=G_p$) 
	or a union of $G$-orbits (forcing $H=G$).  
	Thus $G_p$ is maximal.  
	
	If $p$ and $q$ lie in different orbits but have $G_p=G_q$, then notice that $pg=ph$ if and only if $qg=qh$. We can therefore construct the $G$-partition whose parts are $\{ pg, qg \}$ for $g \in G$  (and singletons outside of $pG \cup qG$). This 
	has a non-singleton part that is not a union of orbits, 
	contradicting (ii).  
	Hence stabilisers of points in different orbits are distinct.
	
	(iii) $\Rightarrow$ (i):  
	Let $\mathcal{P}$ be a $G$-partition and let $P \in \mathcal{P}$ 
	contain distinct points $p$ and $q$.  
	Its setwise stabiliser contains:
	 $\langle G_p, g \rangle$ if $q = pg$, or  
	$\langle G_p, G_q \rangle$ if $p$ and $q$ are in different orbits.
	
	In either case, maximality and distinctness of stabilisers imply that 
	this subgroup is $G$, so $P$ contains the $G$-orbit of $p$.  
	Thus the action is strongly orbit primitive.
	
\end{proof}

The following is now immediate. 

\begin{corollary}
	Suppose that the action of a group $G$ on a set $S$ is strongly orbit primitive. Then the restriction of $G$ to any orbit is primitive. 
\end{corollary}

We note that being strongly orbit primitive is stronger than restricting to a primitive action on each orbit as the following example shows.

\begin{example}
	\label{ex:not strongly orbit primitive}
	
	Consider the action of $\Sym_n \times \Sym_n$ on $2n$ points, where the first factor acts on $\{ 1, \ldots, n\}$ and the second factor acts on $\{n+1, \ldots, 2n\}$. This action has $2$ orbits and is strongly orbit primitive; each stabiliser is maximal and points from distinct orbits have different stabilisers.

	 Now consider the diagonal subgroup, $\Delta = \{ (\sigma, \sigma) \in \Sym_n \times \Sym_n \}$. Note that $\Delta$ also admits $2$ orbits and the restriction to each orbit induces a primitive action. However, while point stabilisers in $\Delta$ are maximal subgroups, the point stabilisers in different orbits can coincide; for example $1$ and $n+1$ have the same stabiliser in $\Delta$. Hence the action of $\Delta$ is primitive when restricted to each orbit but is not strongly orbit primitive.
	\end{example}

\subsection*{Block Systems}

In the context of primitive actions it is useful to use the notion of blocks. We will extend this to intransitive actions. We first recall the definition of a block.

\begin{definition}
	Let $G$ be a group acting on a set $S$, perhaps intransitively. We say that $B \subseteq S$ is a block if $Bg \cap B \neq \emptyset \implies Bg=B$.
\end{definition} 
\begin{remark}
	Note that if we are given a block, $B$, then there is a $G$-partition on $S$ such that $B$ is a part of the partition. 
\end{remark}

We now introduce the notion of a block system. The idea here is that if a group acts on a set and is not strongly orbit primitive, then there is some partition that witnesses the fact. A block system will be such a witness, where a choice has been made for an orbit representative for each part of the partition. 

\begin{definition} \label{blocksystem}
	Suppose $G$ acts on a set $S$. We say that $\{ B_1, \ldots, B_k \}$ is a \emph{block system} for this action if the following conditions hold. 
	
	\begin{enumerate}
		\item Each $B_i$ is a subset of $S$.
		\item Each $G$-orbit meets exactly one of the $B_i$.
		\item For any $g \in G$ and any $1 \leq i \leq k$, $$B_i g \cap B_i \neq \emptyset \implies B_i g =B_i.$$   
	\end{enumerate}
	Furthermore, we say that a block system is \emph{proper} if no $B_i$ contains an orbit and \emph{trivial} if every $B_i$ is a singleton. We say that a block system is \emph{finite} if $|B_i|<\infty$ for $i=1,\ldots, k$.
\end{definition}

It is easy to see that a block system encodes a $G$-partition and vice versa. 

\begin{lemma}
	Let $G$ act on $S$ with a block system, $\{ B_1, \ldots, B_k \}$. Then there exists a congruence, $\sim$ (which is to say a $G$-equivariant equivalence relation on $S$) generated by $\{ B_1, \ldots, B_k \}$  in the following sense: 
	$$
	\omega \sim \delta \iff \omega, \delta \in B_i g \text{ for some } 1 \leq i \leq k, g \in G.
	$$
	
	The sets $B_i$ are then distinct equivalence classes with respect to $\sim$. 
\end{lemma}
\begin{proof}
	Note that condition (ii) of Definition~\ref{blocksystem} implies that every $\omega \in S$ is in some $B_i g$. Hence $\sim$ is reflexive. Symmetry and equivariance of $\sim$ are immediate. 
	
	For transitivity, consider $\omega, \delta, \gamma \in S$ such that $\omega\sim \delta$ and $\delta \sim \gamma$. Then $ \omega, \delta \in B_i g$, and $\delta, \gamma \in B_j h$ for some $i,j$ and $g, h \in G$. Condition (ii) of Definition~\ref{blocksystem} implies that $i=j$, and condition (iii) implies that $B_ig = B_ih$ and hence $\omega \sim \gamma$. 
	
	The final statement is clear, since the sets $B_i$ are disjoint.    
\end{proof}

Conversely, we have the following: 

\begin{lemma}
	Suppose $G$ acts on $S$ with finitely many orbits. Let $\sim$ be a congruence on $S$. Then there exists a block system $\{ B_1, \ldots, B_k \}$ which generates $\sim$. 
\end{lemma}
\begin{proof}
	Since $G$ acts on $S$ with finitely many orbits, there can only be finitely many $G$-orbits of $\sim$-equivalence classes. Simply take one equivalence class per $G$-orbit to get $\{ B_1, \ldots, B_k \}$. It is straightforward to verify that this is a block system which generates $\sim$.  
\end{proof}

\begin{remark}
	In our case where $G\le H_n$ has full Hirsch length and block system $\{ B_1, \ldots, B_k \}$, we may have that $k$ is strictly less than the number of $G$-orbits in $\Ra_n$. This happens precisely when some $B_i$ meets more than one orbit. 
\end{remark}

The following is now clear. 

\begin{lemma}
	Let $G$ act on a set $S$ with finitely many orbits. Then the action is strongly orbit primitive if and only if the action admits no proper, non-trivial block system.  
\end{lemma}

\begin{remark}
	The condition of having finitely many orbits is only for convenience, since that will be our main concern. In general, the number of blocks in a block system could be a cardinal rather than a natural number. 
\end{remark}

\begin{remark}
	If one returns to Example~\ref{ex:not strongly orbit primitive}, we see that $\Delta$ admits a proper non-trivial block system consisting of a single block; $\{ 1, n+1\}$.	This shows how we do not require that each $B_i$ only meets a single orbit. 
\end{remark}

\subsection*{Jordan-Wielandt Theorem}

In order to motivate the discussion, we recall a theorem relating to general permutation groups which may be found in Cameron's text \cite{CameronPG}. Cameron indicates that it is proved using an argument of Wielandt originally devised to apply to finite groups. In Section \ref{sec:cameron} we generalise Theorem \ref{cameron} to intransitive actions.

\begin{theorem}[\cite{CameronPG}, Theorem 6.1]\label{cameron}
	An infinite primitive permutation group which contains a non-identity finitary permutation contains the alternating group.
\end{theorem}
Applying this theorem to the case of the Houghton groups provides information about our Theorem~\ref{main} in the case where $G\le H_n$ acts primitively. In particular, the following corollary shows that such a $G$ must have finite index. Then, since finiteness properties are commensurability invariants, we have that $G$ is of type $\FFF_{n-1}$ but not $\FP_n$. Recall, given any $g\in H_n$, that $t(g)=(t_1(g), \ldots, t_n(g))$ denotes the translation vector of $g$.

\begin{remark*} It is clear that if $G\le_fH_n$, then $G$ has full Hirsch length.
\end{remark*}
\begin{corollary}\label{cameron2} Let $n\in \{2, 3, \ldots\}$ and $G$ be a subgroup of $H_n$ of full Hirsch length. Then $G\le_f H_n$  if and only if $G$ acts primitively on $\Ra_n$.
\end{corollary}
\begin{proof} Assume $G\le_fH_n$. Then $G\cap \Alt(\Ra_n)\le_f \Alt(\Ra_n)$ and, since $\Alt(\Ra_n)$ is infinite and simple (by Lemma \ref{lem:AltisSimple}), it cannot admit any proper finite index subgroup and so $\Alt(\Ra_n)\le G$. The action of $\Alt(\Ra_n)$ on $\Ra_n$ is primitive and therefore so is the action of $G$.
	
	For the converse, assume $G$ acts primitively on $\Ra_n$. Then $G$ has full Hirsch length means that
	$$\quotient{G\FSym(\Ra_n)}{\FSym(\Ra_n)}\le_f\quotient{H_n}{\FSym(\Ra_n)}
	\Leftrightarrow G\FSym(\Ra_n)\le_fH_n.
	$$
	This holds if and only if $G\Alt(\Ra_n)\le_fH_n$, since $\Alt(\Ra_n)$ has index 2 in $\FSym(\Ra_n)$. Hence it is sufficient to show that $\Alt(\Ra_n)\le G$ and, by Theorem \ref{cameron}, we need only show that $G\cap \FSym(\Ra_n)\ne 1$. 
	
	For $n=2$, if $G\cap \FSym(\Ra_n)=1$, then $G\cong \Z$ which contradicts the fact that $G$ acts primitively (for instance, point stabilisers would need to be trivial and hence not maximal). If $n\ge 3$, then $G\cap \FSym(\Ra_n)=1$ implies that $G\cong \Z^{n-1}$ which cannot occur as $\lfloor \frac{n}{2}\rfloor$ is the maximum rank of a free abelian subgroup of $H_n$ \cite[Theorem 3.4]{Simon}. Alternatively, for $n\ge 3$, a subgroup of full Hirsch length contains, for some $d\in \N$, elements $g$ and $h$ such that $t(g)$ and $t(h)$ have just two non-zero entries, specfically $t_1(g)=t_1(h)=d$, $t_2(g)=t_3(h)=-d$. Taking the commutator of $g$ and $h$ produces a non-identity finitary permutation.
\end{proof}

\subsection*{Motivating Example}

The following example is helpful to have in mind, as a lynchpin of the general case. 

\begin{example} We will introduce a particular level subgroup of $H_n$ of full Hirsch length.
Take $\Ra_n=\{1, \ldots, n\}\times \N$, fix some $k\in \N$ and for each $i\in \{1, \ldots, k\}$ define $\N_i:=\{n\in \N : n\equiv i\mod k\}$ and $\Omega_i:=\{1, \ldots, n\}\times \N_i$. Hence $\Ra_n=\Omega_1\sqcup \cdots\sqcup \Omega_k$.

With the notation above, define $\Delta_k:=\bigcap_{i=1}^k \stab_{H_n}(\Omega_i)$. In $H_n$ we have that $\Delta_k$ is the multi-set stabiliser of $\{\Omega_1, \ldots, \Omega_k\}$. We make some further observations.
\begin{itemize}
\item Each $\Omega_i$ is order isomorphic to $\Ra_n$.
\item $\FSym(\Ra_n)\cap \Delta_k=\prod_{i=1}^k \FSym(\Omega_i)$.
\item We have a short exact sequence $1\to \prod_{i=1}^k \FSym(\Omega_i)\to \Delta_k\to (k\Z)^{n-1}\to 1$ where the last map is simply the translation map.
\item The orbits of $\Delta_k$ on $\Ra_n$ are exactly the sets $\Omega_1, \ldots, \Omega_k$.
\item For each $r$, $\Delta_k$ is $r$-transitive on each orbit. 
\end{itemize}
\end{example}

The idea for this example is to take the points on the rays ``modulo $k$'' and look at the stabiliser of those sets. This gives an example of a subgroup of full Hirsch length. Our strategy is to show that, in a certain sense, it is the typical example.

\subsection*{Returning to Houghton Groups}

We now return to Houghton groups acting on the ray system and prove some lemmas about blocks which may arise.

\begin{lemma}\label{4.3} Let $G \leq H_n$ acting on the ray system $\Ra_n$. 
If $B$ is an infinite block for $G$, then $B=B\cdot\gfin$.
\end{lemma}
\begin{proof}
Choose any $f\in \gfin$. For any infinite subset $B$ we have $Bf\cap B\not=\emptyset$. Thus if $B$ is an infinite block, then $Bf=B$.
\end{proof}
For a block $B$ we can consider the \emph{set of translates of} $B$, which is $\{Bg : g\in G\}$. The following states that if $B$ is a finite block, then there can only be finitely many translates of $B$ which meet more than one ray.
\begin{lemma}\label{4.5}
Let $n\in \{2, 3, \ldots\}$, $G\le H_n$ have full Hirsch length and $B$ be a finite block with respect to the action of $G$ on $\Ra_n$. Then the set $$\{Bg : g\in G\textrm{ and }Bg\textrm{ meets more than one ray}\}$$ is finite.
\end{lemma}
\begin{proof} 
The proof with at least 3 rays is somewhat easier so we start with that. So let us first assume that $n \geq 3$. 	
	
If infinitely many translates of $B$ meet more than one ray, then we can find $i\ne j$ such that infinitely many distinct translates of $B$ meet both the $i$th and the $j$th rays.
Then there is an element $y\in G$ with $t_i(y)>0 $ and $t_j(y)=0$. If $Bg$ is a translate of $B$ which meets both the $i$th and $j$th rays sufficiently far out that $y$ is acting by its translation vector then the repeated application of $y$ moves $Bg\cap R_i$ to a disjoint set but fixes $Bg\cap R_j$. This is not possible and therefore only finitely many translates of $B$ can meet both the $i$th and $j$th rays.

Now we prove the result for $n=2$. Choose some $y \in G$ such that $t_1(y) > 0$ (meaning that $t_2(y) = -t_1(y)$). We can then write $y$ as a product of finitely many disjoint cycles. In particular, $\langle y \rangle$ acts with finitely many orbits on $\Ra_2$ some of which are finite and some of which are infinite. Moreover, if $p \in \Ra_2$ is in an infinite $\langle y \rangle$ orbit, then $py^k$ is eventually in the first ray for some large (positive) power of $k$. This implies that a translate of $B$ cannot both contain a point, $q$, from some finite $\langle y \rangle$  orbit and also a point $p$ from some infinite $\langle y \rangle$  orbit, since in this case we would be able to find a positive integer $k$ such that $q y^k = q$ and $py^k$ is arbitrarily far in the first ray and hence not an element of this $B$ translate. Moreover, only finitely many translates of $B$ are contained in the finite $\langle y \rangle$ orbits. 

We will therefore restrict to considering translates of $B$ which only meet infinite $\langle y \rangle$  orbits. Let us argue by contradiction and suppose that there are infinitely many of these that meet both rays. Since $\langle y \rangle$  acts with finitely many orbits, we may also assume that these $B$ translates are all in the same $\langle y \rangle$ orbit. But now, without loss of generality, $B$ meets both rays and contains only points from infinite $\langle y \rangle$  orbits and we have infinitely many integers $m$ such that $By^{m}$ also meets both rays. However, for sufficiently large $k$, $By^k$ is contained in either the first ray (for large positive $k$) or the second ray (for large negative $k$). This contradiction proves the result for $n=2$.
\end{proof}

Note that, from the reduction in the previous section, we can assume that no points of $\Ra_n$ lie in a finite orbit of any group $G\le H_n$ of full Hirsch length.

\begin{lemma}\label{4.6}
Let $n\in \{2, 3, \ldots\}$, $G\le H_n$ have full Hirsch length and define $e:=|H_n:G\FSym|$. If $B$ is a finite block (with respect to the action of $G$ on $\Ra_n$) whose elements do not lie in a finite orbit of $G$, then $B$ contains at most $e$ elements.
\end{lemma}
\begin{proof} 
Since no elements of $B$ lie in a finite orbit of $G$, it has infinitely many translates. By Lemma \ref{4.5}, there is a ray $R_j$ which entirely contains infinitely many translates. Choose $y\in G$ so that $t_j(y)$ is a divisor $d>0$ of $e$. We shall show that $B$ has at most $d$ elements, from which the result follows.

Let $z\in \N$ be chosen so that $(j, m)y=(j, m+d)$ for all $m\ge z$, and let $S:=\{(j, m) : m\ge z\}$. Replacing $B$ by a suitable translate, we may assume that $B\subset S$. If $B$ has more than $d$ elements, then there must be two that are congruent modulo $d$: that is, there exist $p, q\in \N$ and $m\ge1$ such that $(j, p), (j, q)\in B$ and $q=p+md$. But then $(j, q)=(j, p)y^m\in B\cap By^m$, contradicting the definition of a block.
\end{proof}

Suppose $\mathcal{B} = \{ B_1, \ldots, B_k\}$ is a proper block system  for $G$ where every block, $B_i$, is a finite subset of $\Ra_n$. Let $\sim$ be the $G$-congruence generated by $\sim$. Define an ordering on $\Ra_n/\sim$ by $B_ig<_{\Ra_n/\sim} B_jh \Leftrightarrow \min(B_ig)<_{\Ra_n} \min(B_jh)$. Recall Lemma \ref{4.5} which stated that, except for a finite number of exceptions, for each set $B_ig$ (where $i\in \{1, \ldots, k\}$ and $g\in G$) there exists a $j\in \{1, \ldots, n\}$ such that $B_ig\subset R_j\subset \Ra_n$.

\begin{proposition}
	\label{ordering blocks}
	Let $n\in \{2, 3, \ldots\}$, $G\le H_n$ and $\Ra$ be the ray system for $H_n$. Suppose that $G$ admits a proper block system, whose blocks are finite and which generates the $G$-congruence, $\sim$. Then:
	\begin{itemize}
		\item[(i)] $\Ra/\sim$ (with the ordering defined above) is order isomorphic to $\Ra$ (with the lexicographic ordering).
		\item[(ii)] Let $\rho: \Sym(\Ra)\to \Sym(\Ra/\sim)$ be induced by the natural map $\Ra \to \Ra/\sim$. Define $\Gamma:=\rho(G)$. Then $\Gamma$ is a subgroup of the Houghton group of $\Ra/\sim$.
	\end{itemize}
\end{proposition}
\begin{proof}
We start with (i). This is similar to Lemma \ref{lem:orderiso1} due to Lemma \ref{4.5}. Let 
$$S_1:=\{B_jg \mid j\in \{1, \ldots, k\}, g\in G\textrm{ and }B_jg\textrm{ meets }R_1\}$$
Then $S_1$ has a least element, which is the block $B$ containing $(1,0)$ and every element in $S_1\setminus\{B\}$ and point in $R_1\setminus\{(1, 0)\}$, with their respective orderings, is a successor element (as in Fact \ref{orders2}). Sending $B\in S_1$ to $(1, 0)\in R_1$ defines a partial order isomorphism $f_1: S_1 \to R_1$ sending the initial segments of $S_1$ to the initial segments of $R_1$. Note that $S_1$ is infinite since all of our blocks are finite. Thus we can choose $f_1$ to be a function (that is, to have domain $S_1$) and because $R_1$ has the same order type as $\N$, our function $f_1$ must be surjective. Hence $f$ is an order preserving bijection between $S_1$ and $R_1$.

Let $Y_1$ consist of all of the points of $\Ra$ contained in $S_1$, i.e., the union of all of the blocks in $S_1$. Then, by Lemma \ref{4.5}, $Y_1\setminus R_1$ is finite. Hence $\Ra\setminus Y_1$ and $\Ra\setminus R_1$ are order isomorphic by Lemma \ref{lem:orderiso1}. With notation indicative of repeating the argument in the first paragraph, for $i=2, \ldots, n$ we define 
$$S_i:=\{B_jg \mid j\in \{1, \ldots, k\}, g\in G\textrm{ and }B_jg\textrm{ meets }R_i\}\setminus (\bigcup\nolimits_{d<i} S_d)$$ and $Y_i\subset \Ra$ to be the union of the blocks in $S_i$. As before, Lemma \ref{4.5} states that $Y_i\setminus R_i$ is finite for $i=2, \ldots, n$, and, because all our blocks are finite, that $S_2, \ldots, S_n$ are all infinite. Also, $S_2$ again contains a least element (though in this case the block containing $(2,0)$ may lie in $S_1$). The same argument in the preceding paragraph provides an order isomorphism $f_2$ between $S_2$ and $R_2$ induced by sending the least element of $S_2$ to $(2, 0)\in R_2$. By continuing in this manner, we obtain order isomorphisms $f_i: S_i\to R_i$ for $i=1, \ldots, n$. Observe that $\Ra/\sim=\bigcup_{i\le n}S_i$, and so we may define $f: \Ra/\sim\to \Ra$ by the image of $f_i$ on $S_i$ for $i=1, \ldots, n$. This completes the construction of the required order isomorphism.

We now appeal to Proposition \ref{aopisHou} to deduce (ii). Any element that is not almost order preserving on $\Ra/\sim$ will (for any lift to $\Ra$) produce an element that is not almost order preserving on $\Ra$. Since $G$ is a subgroup of $H$ acting on $\Ra$, this means that $\Gamma$ can only consist of elements that are almost order preserving of $\Ra/\sim$, and hence $\Gamma$ is a subgroup of the Houghton group acting on $\Ra/\sim$.
\end{proof}
\section{The multi-wreath product}\label{sec:multiwreath}

In this section we would like to give a small modification of a permutational wreath product, that will better suit our purposes (involving non-transitive actions).

Let $\Gamma$ be a group and let $\Omega$ be a right $\Gamma$-set with finitely many orbits $\Omega_1, \ldots, \Omega_k$. Given groups $W_1, \ldots, W_k$ (that is one group for each orbit), let $\mathcal{W}:= W_1 \sqcup \cdots \sqcup W_k$. We then build the \emph{restricted (resp. unrestricted)} permutational wreath product $\mathcal{W}\wr_\Omega \Gamma$ (resp.  $\mathcal{W}\Wr_\Omega \Gamma$) as a semidirect product $\Phi\rtimes \Gamma$ where $\Phi$ is the set of functions from $\Omega$ to $\mathcal{W}$ \emph{with finite support (resp. with arbitrary support)} such that the restriction of such a function to $\Omega_i$ lies in $W_i$. We have that
$$\Phi=\{\phi: \Omega \to \mathcal{W} \mid \phi(\omega)\in W_i\text{ for all }\omega\in \Omega_i\}$$
and $\Phi$ forms a group with pointwise multiplication.

Furthermore, $\Gamma$ acts on $\Phi$ on the right according to the rule
$$\phi^a(\omega)=\phi(\omega a^{-1}),$$
for $\omega\in\Omega$, $a\in \Gamma$ and $\phi\in\Phi$. Note that the right action on $\Omega$ forces a left action on $\Phi$ which we turn into a right action by using the inverse.
Elements of the semidirect product are ordered pairs $(\phi, a)$ with multiplication defined by
$$(\phi_1, a_1)(\phi_2, a_2)=(\phi_1\phi_2^{a_1^{-1}}, a_1a_2).$$
We refer to $\Phi$ as the \emph{base} of the wreath product, and to $\Gamma$ as the \emph{head}. We identify $\Phi$ and $\Gamma$ with their images in the wreath product. We note that for a standard permutational wreath product one simply sets all the $W_i$ to be the same group $W$. In fact, our wreath product is a subgroup of the standard permutational one by setting $W:=\bigoplus_{i=1}^kW_i$.

Thus there are two kinds of wreath product and clearly the restricted wreath product is naturally a subgroup of the unrestricted wreath product. Their relationship is closely tied to the relationship in homological algebra between the induced module and the coinduced module. The Kaloujnine--Krasner theorem concerning embedding certain permutation groups in wreath products can be interpreted as a form of non-abelian second cohomology calculation. We look at this theorem next.

\subsection*{A formulation of the Kaloujnine--Krasner Theorem}
Suppose that $G$ is a group and $\Omega$ is a right $G$-set. Assume that $G$ acts faithfully on $\Omega$. That is, the action is given by an \textit{injective} homomorphism, $G \to \Sym(\Omega)$. The following notions will be helpful.

\begin{definition}
For $X\subset \Ra_n$ and $K\le H_n$, we write $\stab_K\{X\}$ for its setwise stabilizer
$$\stab_K\{X\}:=\{k\in K;\ Xk=X\}$$
and $\stab_K(X)$ for its pointwise stabilizer
$$\stab_K(X):=\{k\in K;\ \forall x\in X,\ xk=x\}.$$
Furthermore, for a block $B$ and $K$ a subgroup of $G$, we can consider the subgroup $W(K)$, the group of automorphisms of $B$ induced by $K$ given by
$$W(K)=\frac{K\cap\stab_G\{B\}}{K\cap\stab_G(B)}.$$
Note, given $K\le L\le G$, we have that there is a natural inclusion $W(K)\subseteq W(L)$.
\end{definition}

\begin{notation} We now suppose that $G$ acts faithfully and with finitely many orbits on a set $\Ra$ and that $\sim$ is a congruence generated by a block system, $\{B_1, \ldots , B_k\}$. Let $W_i=W_i(G)$ denote the group of permutations of $B_i$ induced by $G$: namely, $W_i=\frac{G\cap\stab\{B_i\}}{G\cap\stab(B_i)}$, as above. Then let $\mathcal{W} = W_1 \sqcup \cdots \sqcup W_k$. Define $\Omega := \Ra/\sim$ and {$\Omega_i:=B_iG$ to be the $G$-orbits of $\Omega$ so that $\Omega=\bigsqcup_{i=1}^k\Omega_i$}. The map $\Ra \to \Omega$ gives rise to a homomorphism $G \to \Sym(\Ra) \to \Sym(\Omega)$; denote $G \to \Sym(\Omega)$ by $\rho$. Let $\Gamma=\im(\rho)$. As above, $\Phi=\{\phi: \Omega \to \mathcal{W} \mid \phi(\omega)\in W_i$ for all $\omega\in \Omega_i\}$.
\end{notation}

With this setup, $G$ is isomorphic to a subgroup of the unrestricted permutational wreath product $\mathcal{W}\Wr_\Omega \Gamma$.

To make the isomorphism explicit, let $g\mapsto \rho(g)$ denote the natural map $G\to \Gamma$, as above. Observe that $\phi^g=\phi^{\rho(g)}$ for all $g\in G$ and for all $\phi\in \Phi$. Then we choose a map
$$
\tau:\Omega\,\,\to G
$$

which satisfies

$$
B_i \tau(\omega) = \omega\text{ for }\omega \in \Omega_i.
$$

That is, for every $\sim$ equivalence class, we choose a group element which maps one of the $B_i$ to it. Note for any $g\in G$ and $\omega \in \Omega_i$ that $B_i\tau(\omega)g=\omega g$ hence $B_i\tau(\omega)g\tau(\omega g)^{-1}=B_i$. Hence $\tau(\omega)g\tau(\omega g)^{-1}$ induces an element of $W_i$ and so we define $\phi_g\in \Phi$ (so that $\phi_g:\Omega\,\to \mathcal{W}$) by 

$$
 \phi_g(\omega):= \tau(\omega)g \tau(\omega g)^{-1}
$$
for any $g\in G$. Note that we are abusing notation since the right-hand-side is a coset of $\stab_G(B_i)$ when $\omega \in \Omega_i$. 

We can now define a map

$$\kappa:G\to W\Wr_\Omega \Gamma=\Phi\rtimes \Gamma$$ by
$$g\mapsto(\phi_g,\rho(g)).$$

\begin{theorem}[Kaloujnine--Krasner]\label{kk} The homomorphism
$\kappa$ so defined is injective.
\end{theorem}

\begin{proof} First we check that $\kappa$ is a group homomorphism.
We need to show that for all $g, h\in G$,
$$(\phi_{gh}, \rho(gh))=(\phi_{g}, \rho(g))(\phi_{h}, \rho(h))$$
and this reduces to
$$\phi_{gh}=\phi_{g}\phi_{h}^{g^{-1}}\eqno(6.1)$$
and $$\rho(gh)=\rho(g)\rho(h).\eqno(6.2)$$
Equation (6.2) is immediate as $\rho$ is a homomorphism. For (6.1),
we have
$$\phi_{gh}(\omega)=\tau(\omega) gh\tau(\omega gh)^{-1}$$
and so
$$\begin{array}{rcl}
\phi_g\circ \phi_h^{g^{-1}}(\omega)&=&\phi_g(\omega)\phi_h^{g^{-1}}(\omega)\\
&=&\tau(\omega)g\tau(\omega g)^{-1}\tau(\omega g)h \tau(\omega gh)^{-1}\\
&=&\tau(\omega)gh \tau(\omega gh)^{-1}\\
&=&\phi_{gh}(\omega)
\end{array}$$

Next, for injectivity, if $\rho(g)=1$, then $\omega g=\omega$ for all $\omega\in \Omega$ and so $\phi_g(\omega) = \tau(\omega)g\tau(\omega)^{-1}$. Hence, if $g$ is in the kernel of $\kappa$, then $g$ acts trivially on $\omega$ for all $\omega\in\Omega$ and since the action of $G$ on $\Ra$ was assumed to be faithful, it follows that the kernel of $\kappa$ is trivial.
\end{proof}

\medskip

The idea now is to choose $\tau$ more carefully than above. More precisely, we can choose $\tau$ so that if $g$ induces an order preserving bijection from $\omega$ to $\omega g$ (which it will do for almost all $\omega\in \Omega_i$ when we are considering Houghton groups), then $\phi_g(\omega)$ will be the identity element of $W_i$.

\begin{lemma}
	\label{restricted}
If $G$ is a group of almost order preserving bijections of the totally ordered set $\Ra$, then the image of $\kappa$ can be taken into the restricted wreath product.
\end{lemma}
\begin{proof}
For any $g \in G$, $g$ preserves the order on all but finitely many points of $\Ra$. Hence for all but finitely many $\omega \in \Omega$, the map $g:\omega \to \omega g$ is an order preserving bijection. Fix some $i\in \{1, \ldots, k\}$. Set $B:=B_i$, $n:=|B|$ and choose a linear order on $S_{n}$ such that the identity is minimal. Therefore $B=\{b_1, \ldots, b_n\}$ where $i<j$ if and only if $b_i<b_j$. Now choose $\omega \in BG$. Using the ordering from $\Ra$, we have $\omega=\{\delta_1, \ldots, \delta_n\}$ where again $i<j$ if and only if $\delta_i<\delta_j$. For any given choice of $\tau$, there exists a unique $\sigma_\omega\in S_n$ such that
$$b_i\tau(\omega)=\delta_{\sigma_\omega(i)}.$$
Choose each $\tau(\omega)$ so that the corresponding $\sigma_\omega$ is minimal in $S_n$ with respect to our given ordering. We claim that with this choice of $\tau$, the image of $\kappa$ is into the restricted wreath product. For all but finitely many $\omega$, $g$ is order preserving when restricted to $\omega$. Therefore
$$b_i\tau(\omega)g=\delta_{\sigma_\omega(i)}g.$$
Now, because $\omega g=\{\epsilon_1, \ldots, \epsilon_n\}$ where $i<j$ if and only if $\epsilon_i<\epsilon_j$ and $\epsilon_i=\delta_i g$, hence
$$b_i\tau(\omega g)=\epsilon_{\sigma_{\omega g}(i)}$$
and also 
$$b_i\tau(\omega)g=\delta_{\sigma_\omega(i)}g=\epsilon_{\sigma_{\omega}(i)}.$$
By the minimality of $\sigma_{\omega g}$ we get that $\sigma_\omega\ge \sigma_{\omega g}$ and by symmetry $\sigma_\omega\le \sigma_{\omega g}$. Hence $\sigma=\sigma_\omega=\sigma_{\omega g}$ and
$$\begin{array}{rcl}
b_i\phi_g(\omega)&=&b_i\tau(\omega)g\tau(\omega g)^{-1}\\
&=&\delta_{\sigma(i)}g\tau(\omega g)^{-1}\\
&=&\epsilon_{\sigma(i)}\tau(\omega g)^{-1}\\
&=&b_{\sigma^{-1}\sigma(i)}=b_i
\end{array}$$
meaning that $\phi_g$ induces the identity on all but finitely many $\omega$.
\end{proof}
Our final aim in this section is to show that the image of $\kappa$ is finite index in $\mathcal{W}\wr_\Omega \Gamma$. The following will be helpful.
\begin{notation} \label{wreath notation} We recall the notation relating to the above setup.
\begin{enumerate}[(i)]
\item Let $\base$ denote the base of our restricted multi-wreath product $\mathcal{W}\wr_\Omega \Gamma$. 
\item Let $\rho: G \to \Gamma$ be the composition of the function $\kappa$ with the function $\mathcal{W}\wr_\Omega \Gamma\to \Gamma$ obtained by taking the quotient by $\base$. By construction, $\rho(G) = \Gamma$. 
\item Let $N$ denote the kernel of $\rho$. Equivalently, $N$ is the kernel of the induced action of $G$ on $\Ra/\sim = \Omega$. 
\end{enumerate}
\end{notation}
Our aim is now to prove the following. It will be shown in the next section that the additional hypotheses (i)-(iii) can be satisfied for any $n\in\{3, 4, \ldots\}$ and any level subgroup $G$ of $H_n$ whose orbits are all infinite.
\begin{theorem}\label{thm-kappaG} Suppose that $G$ acts on a set $\Ra$ all of whose orbits are infinite and with a block system, $\{ B_1, \ldots, B_k\}$ generating a $G$-equivalence relation, $\sim$. Let $\Omega= \Ra/\sim$. Suppose that, as above, $\kappa$ is an injective group homomorphism from $G$ to the multi-wreath product $\mathcal{W}\wr_\Omega \Gamma$ (that is, we assume that the image lands in the restricted wreath product). 

Suppose further that, 	
	\begin{enumerate}
		\item $|B_i|< \infty$ for $i=1, \ldots, k$.
		\item \label{induced perm} $W_i(G)=W_i(N)$ for $i=1, \ldots k$.
		\item $\Alt(\Omega_1)\times\cdots\times\Alt(\Omega_k)\le \Gamma$ when $\Gamma$ is viewed (in the natural way) as a subgroup of $\Sym(\Omega_1)\times\cdots\times\Sym(\Omega_k)$.
	\end{enumerate}

Then $\kappa(G)$ has finite index in $\mathcal{W}\wr_\Omega \Gamma$.

\end{theorem}

To show that $\kappa(G)$ has finite index in $\mathcal{W}\wr_\Omega \Gamma$, we will exhibit a finite subset of $\mathcal{W}\wr_\Omega \Gamma$ whose product with $\kappa(G)$ is the whole of $\mathcal{W}\wr_\Omega \Gamma$. In fact, our finite subset will be a subset of $\Phi$, the base of the wreath product. We shall do this in stages.

\begin{definition} \label{support} Recall that elements of $\base$ are functions from $\Omega$ to $\mathcal{W} = W_1 \sqcup \ldots \sqcup W_k$. Let $1_{W_i}$ denote the identity element of $W_i$. 
	\begin{enumerate}[(i)]
		\item  Given $\phi \in \Phi$, we define the support of $\phi$ to be the set $$\supp(\phi) = \{ \omega \in \Omega \ : \ \phi(\omega) \neq 1_{W_i} \text{ for any } i \}.$$
		\item  	Let $S \subseteq \Omega$ be finite. We define $\Phi_S$ to be the full subset of $\Phi$ whose elements have support contained in $S$. This is a finite subgroup of $\Phi$. 
	\end{enumerate}	
 
\end{definition}

\begin{lemma} \label{subset F}
	Given the hypotheses of Theorem~\ref{thm-kappaG}, there exists a finite subset, $F \subseteq N \subseteq G$ such that $W_i(F) = W_i(N)=W_i(G)$ for each $i$. 
\end{lemma}
\begin{proof}
	By Theorem~\ref{thm-kappaG} (\ref{induced perm}), $W_i(G) = W_i(N)$. This means that if some element $g \in G$ induces a permutation on some $B_i$, then there is an $n \in N$ which induces that same permutation on $B_i$. Since there are only finitely many $B_i$ and each of these is finite, we can produce a finite subset $F$ of $N$ such that any permutation induced by any element of $G$ on some $B_i$ is equal to that induced by some element of $F$.  
\end{proof}

\begin{proposition}
	\label{prop:coset reps}
	Let $F$ be the finite subset given by Lemma~\ref{subset F}. Let $S = \bigcup_{f \in F} \supp(f)$ and let $\Phi_S$ be the (finite) subgroup of $\Phi$ whose supports lie in $S$ (as in Definition~\ref{support}). 
	
	Then $\Phi_S \kappa(G)   = \mathcal{W}\wr_\Omega \Gamma$. 
\end{proposition}

\begin{proof}
We first note that for any $\alpha \in \mathcal{W}\wr_\Omega \Gamma$, $\alpha \kappa(G) \cap \Phi \neq 1$. This follows immediately from Notation~\ref{wreath notation} (ii), since $\rho(G) = \Gamma$. 

Therefore, for each $\alpha \in \mathcal{W}\wr_\Omega \Gamma$ we can define the following natural number: 

$$
n_S(\alpha \kappa(G)) := \min  \{ |\supp(\alpha_0) \setminus S| \ : \ \alpha_0 \in  \alpha \kappa(G) \cap \Phi \}.  
$$
That is, we look at the coset $\alpha \kappa(G)$ and for each element of the coset which also lies in the base, we look at the size of the support, once we have removed the finite subset $S$ (and then take the minimum over all these natural numbers). It is clear that it is sufficient to prove that $n_S(\alpha \kappa(G)) = 0 $ for every $\alpha \in \mathcal{W}\wr_\Omega \Gamma$. 

Let us argue by contradiction and suppose that we have a coset $\alpha \kappa(G)$ such that $n_S(\alpha \kappa(G)) > 0 $. Choose some $\alpha_0 \in  \alpha \kappa(G) \cap \Phi$ which realises this minimum. Hence there exists some $\omega \in \Omega_i \setminus S$ such that $\alpha_0(\omega) \neq 1_{W_i}$ (note that $\omega$ is in \emph{some } $\Omega_i$).

Recall that for our finite subset, $F$,we have $W_i(F) = W_i(N)=W_i(G)$. This means that any permutation of $B_i$ induced by some element of $G$ is also induced by some element of $F$. Equivalently, we can write this as

$$
\{ \kappa(f) (B_i) \ : \ f \in F \} = W_i. 
$$

Moreover, the union of the supports of the $\kappa(f)$ lie in $S$. 

Next we claim that there exists some $\gamma \in \Gamma$ such that:
\begin{itemize}
	\item $\omega  = B_i \gamma$ and 
	\item $S \gamma \subseteq S \cup \{ \omega \}$. 
\end{itemize}

The existence of this element is guaranteed as $\Gamma$ contains the product of alternating groups on each $\Omega_i$ and that each $\Omega_i$ is infinite - this is one of the hypotheses, Theorem~\ref{thm-kappaG} (iii).

Moreover, this implies that 

$$
\{ \kappa(f)^{\gamma} (\omega) \ : \ f \in F \} = W_i. 
$$
and that the union of the supports lie in $S \gamma \subseteq S \cup \{ \omega\}$. 

Finally, we consider an element $g \in G$ such that $\kappa(g) = \gamma$, which must exist since $\rho(G) = \Gamma$. Notice that since $\gamma^{-1} \kappa(g) \in \Phi$ (the base of the wreath product) we must get that, 
$$
\{ \kappa(f^g) (\omega) \ : \ f \in F \} = W_i. 
$$
Moreover, $\supp \kappa(f^g) \subseteq S \cup \{ \omega\}$. Hence there exists some $f \in F$ such that $\kappa(f^g)(\omega) = \alpha_0(\omega)$ and $\supp \kappa(f^g) \subseteq S \cup \{ \omega\}$. But this implies that 
$$
|\supp(\alpha_0 \kappa(f^g)^{-1}) \setminus S| < |\supp(\alpha_0) \setminus S|, 
$$
contradicting the minimality of $\alpha_0$.
\end{proof}

\begin{proof}[Proof of Theorem~\ref{thm-kappaG}] 
	This follows immediately from Proposition~\ref{prop:coset reps}. 
\end{proof}

\section{Subdirect Products}\label{sec:subdirect}

Many of our arguments will utilise the properties of subdirect products, which arise naturally when considering intransitive actions. 

\begin{definition}
	Given groups $\Gamma_1, \ldots, \Gamma_k$, let $\Gamma:=\Gamma_1 \times \cdots \times\Gamma_k$. For $i=1, \ldots, k$, define $\pi_i: \Gamma \to \Gamma_i$, $(\gamma_1,\ldots, \gamma_k)\mapsto \gamma_i$, the $i$th \emph{projection map} on $H$.
	
	A subgroup $G\le \Gamma$ is called a \textit{sub-direct product} of $\Gamma$ if $\pi_i(G)=\Gamma_i$ for $i=1,\ldots, k$. 
\end{definition}

\begin{remark}
	When considering a subdirect product $G$ of $\Gamma_1 \times \cdots \times\Gamma_k$, we will think of $\Gamma_i$ to be a subgroup of  $\Gamma_1 \times \cdots \times\Gamma_k$ in the natural way. Note that it is also a quotient but there should be no confusion when we write $G \cap \Gamma_i$. 
\end{remark}

A very useful property of sub-direct products is the following.

\begin{proposition}
	\label{normal in subdirect}
	Let $G$ be a sub-direct product of $\Gamma_1\times\cdots\times\Gamma_k$. If $N$ is a normal subgroup of $G$, then for each $i$ we have that $N \cap \Gamma_i$ is a normal subgroup of $\Gamma_i$. 
	
	In particular, $G \cap \Gamma_i$ is a normal subgroup of $\Gamma_i$ for $i=1, \ldots, k$.
\end{proposition}
\begin{proof}
	Since $\Gamma_i\unlhd \Gamma_1\times\cdots\times\Gamma_k$, it is clear that $N_i:= N \cap \Gamma_i$ is a normal subgroup of $G$. Let $w \in \Gamma_i$. Then, since $G$ is a sub-direct product, there exists a $g \in G$ such that $\pi_i(w)=\pi_i(g)$. This implies that $w^{-1} g \in \ker \pi_i$. Now, since $\ker \pi_i$ centralises $\Gamma_i$, we get that $N_i^w=N_i^g=N_i$. Hence $N_i$ is also normal in $\Gamma_i$.
	
	The final statement is clear since $G$ is a normal subgroup of itself.   
\end{proof}

The relevance of sub-direct products arises from actions on sets via the following result. 

\begin{proposition}
	\label{actions and subdirect}
	Let $G$ be a group acting faithfully on a set $X$. Suppose that $G$ acts with finitely many orbits, $X_1, \ldots, X_k$. Then $G$ is isomorphic to a sub-direct product of $\Gamma_1 \times \cdots \times \Gamma_k$, where each $\Gamma_i$ acts faithfully on $X_i$. 
\end{proposition}
\begin{proof}
	The (faithful) action of $G$ on $X$ is given by an (injective) homomorphism, $\phi: G \to Sym(X)$. For each $i$, we can restrict the action of $G$ to the orbit $X_i$ to get a homomorphism, $\phi_i: G \to Sym(X_i)$ (this need no longer be injective). Define $\Gamma_i$ to be the image of $\phi_i$ and define a homomorphism from $G$ to $\Gamma_1 \times \cdots \times \Gamma_k$ by, 
	$$
	g \mapsto (\phi_1(g), \phi_2(g), \ldots, \phi_k(g)).
	$$ 
	It is then straightforward to verify that this map is injective (since $\phi$ was injective) and that $G$ is a therefore a sub-direct product as claimed. The fact that $\Gamma_i$ acts faithfully on $X_i$ is a simple consequence of the fact that $\Gamma_i$ is a subgroup of $\Sym(X_i)$. 
\end{proof}

\begin{remark}
	In the aforementioned situation, $G$ is a subgroup (a sub-direct product) of $\Gamma_1 \times \cdots \times \Gamma_k$ which is in turn a subgroup of $\Sym(X_1) \times \cdots \times \Sym(X_k) \leq \Sym(X)$. 
\end{remark}

\section{Generalising a result of Jordan and Wielandt via Cameron}\label{sec:cameron}
From our notion of a strongly orbit primitive action, we can generalise Theorem \ref{cameron}. We do this in Theorem~\ref{prodalt}. After proving our generalisation, we will  show that the hypotheses (i)-(iii) of Theorem~\ref{thm-kappaG} are satisfied for (level) subgroups of Houghton's group of full Hirsch length, and we use this to conclude our structure theorem (Theorem~\ref{thm:structure}).

\begin{theorem}[Theorem 8.2A and Exercise 8.2.1, \cite{DixonMortimerPG}]
	\label{autos of perm groups}
	Let $\Omega$ be an infinite set. Suppose that we have a subgroup  $\Alt(\Omega) \leq H \leq \Sym(\Omega)$. Then the natural map from the normaliser of $H$ in $\Sym(\Omega)$ to $\Aut(H)$ is an isomorphism. 
\end{theorem}

\begin{lemma}
	\label{normal subs of perm groups}
	Let $G$ be a subgroup of $\Sym(\Omega)$ for some infinite set $\Omega$ and suppose that $\Alt(\Omega) \leq G$. Then any non-trivial normal subgroup of $G$ contains $\Alt(\Omega)$.
\end{lemma}
\begin{proof}
	Let $N$ be a non-trivial normal subgroup of $G$. Then $[N, \Alt(\Omega)]$ is a non-trivial normal subgroup of $N \cap \Alt(\Omega)$. (Non-triviality follows since the centraliser of $\Alt(\Omega)$ in $\Sym(\Omega)$ is trivial). 
	
	Since $\Alt(\Omega)$ is simple, the result follows.
\end{proof}

\begin{lemma}
	\label{non trivial block system}
	Let $\Gamma$ act on a set $\Omega$ with finitely many orbits, $\Omega= \Omega_1 \sqcup \ldots \sqcup \Omega_k$, each of which is infinite. Let $\Gamma_i$ be the restriction of $\Gamma$ to $\Omega_i$ so that $\Gamma$ is a sub-direct product of
	$$
	\Gamma_1 \times \cdots \times \Gamma_k \leq  \Sym(\Omega_1) \times \cdots \times \Sym(\Omega_k) \leq \Sym(\Omega).
	$$
	Suppose further that for some $ i \neq j$, 
	\begin{itemize}
		\item $\Alt(\Omega_i) \leq \Gamma_i$ and $\Alt(\Omega_j) \leq \Gamma_j$
		\item $\Gamma \cap (\Gamma_i \times \Gamma_j) \neq 1$
		\item $\Gamma \cap \Gamma_i = \Gamma \cap \Gamma_j = 1$.
	\end{itemize}
	Then the action $\Gamma$ on $\Omega$ admits a proper, non-trivial block system. Hence the action of $\Gamma$ on $\Omega$ is not strongly orbit primitive. 
\end{lemma}
\begin{proof}
	We will denote the projection maps by $\pi_i: \Gamma \to \Gamma_i$ and note that these are surjective by construction. 
	
	We also denote the subgroup $N=\Gamma \cap (\Gamma_i \times \Gamma_j) \neq 1$. Note that this is a normal subgroup of $\Gamma$. Hence $\pi_i(N)$, $\pi_j(N)$ are normal subgroups of $\Gamma_i$ and $\Gamma_j$, respectively. By hypothesis, $\Gamma \cap \Gamma_i = \Gamma \cap \Gamma_j = 1$ and hence $N \cap \Gamma_i=N \cap \Gamma_j =1$. Since $\pi_i(N) \cong N/(N \cap \Gamma_j) \cong N$, we get that $\pi_i(N)$ is a non-trivial normal subgroup of $\Gamma_i$ and must therefore contain $\Alt(\Omega_i)$, by Lemma~\ref{normal subs of perm groups}. By symmetry, $\pi_j(N)$ contains $\Alt(\Omega_j)$.

	Therefore, by Goursat's Lemma, $N$ is the graph of an isomorphism between $\pi_i(N)$ and $\pi_j(N)$. For completeness, we describe the argument below. Write 
	 $$
	 N = \{ (n_i, \rho^*(n_i)) \ : \ n_i \in \pi_i(N)\}. 
	 $$
	 
	 The fact that $N \cap \Gamma_i=1$ implies that $\rho^*: \pi_i(N) \to \pi_j(N)$ is a function. And the fact that $N \cap \Gamma_j =1$ implies that $\rho^*$ is injective. The fact that the image of $\rho^*$ is $\pi_j(N)$ follows by definition. The fact that $N$ is a subgroup now implies that $\rho^*$ is a homomorphism and hence an isomorphism. 
	 
	 We now invoke Theorem~\ref{autos of perm groups} to deduce that there is a $\rho\in \Sym(\Omega)$ which restricts to a bijection from $\Omega_i$ to $\Omega_j$ so that $\rho^*(n_i) = \rho n_i \rho^{-1}$ for all $n_i \in \pi_i(N)$.
	 
	Pick some $\omega \in \Omega_i$ and define $B = \{ \omega, (\omega)\rho \}$. For any $i, j \neq k$ we let $B_k$ be any singleton set containing an element from $\Omega_k$. We claim that $\B:=\{B, B_k \}$ for $i,j \neq k$ is a block system for the action of $\Gamma$ on $\Omega$. Once we prove this claim we have the result since this is clearly non-trivial and proper. In fact the only condition from Definition~\ref{blocksystem} that is not immediate is the third condition for $B$. Namely, we need to check that $Bx \cap B \neq \emptyset \implies Bx = B$. 
	 
	 To verify the claim, we argue as follows. Define $N_{\omega}=\stab_N(\omega)$ to be the stabiliser of $\omega$ in $N$. 
	 
	 Then, for any $f\in \stab_\Gamma(\omega)$, $g \in \stab_{\Gamma}((\omega)\rho)$ and $x\in \Gamma$, we get that:
	 \begin{enumerate}
	 	\item $(N_\omega)^f=N_\omega$, $(N_\omega)^g=N_\omega$
	 	\item $\fix(N_\omega) \cap (\Omega_i \cup \Omega_j)\supseteq \{\omega, (\omega)\rho\}$
	 	\item $\fix(N_\omega) \cap (\Omega_i \cup \Omega_j) = \{\omega, (\omega)\rho\}$
	 	\item $\fix(N_\omega^x)=\fix(N_\omega)\cdot x$
	 	\item $B f\cap B = B$, $Bg \cap B=B$ and
	 	\item $B h\cap B = \emptyset$ if $h\in \Gamma\setminus \stab_\Gamma(\omega)$ or $h \in \Gamma\setminus \stab_\Gamma((\omega)\rho)$ .
	 \end{enumerate}
	 Statements (i), (ii), (iv) are clear. For (iii) we note given any $a, b, c \in \Omega_i\setminus\{\omega\}$ that $((a\;b\;c), (a\;b\;c)\rho)\in N_{\omega}$ and so $\Omega_i\cap \fix(N_\omega)=\{\omega\}$. A symmetric argument using $x, y, z\in \Omega_j\setminus\{(\omega)\rho\}$ shows that $\Omega_j\cap \fix (N_{\omega}) =\{(\omega)\rho\}$. Now (v) follows by combining the other observations, since
	 $$\{\omega, (\omega)\rho\}=\fix(N_\omega) \cap (\Omega_i \cup \Omega_j)=\fix((N_\omega)^f \cap (\Omega_i \cup \Omega_j))=\{\omega \cdot f, (\omega\rho)\cdot f\}$$
	 and so $\{\omega, (\omega)\rho\}=\{\omega, (\omega\rho)\cdot f\}$, i.e., $(\omega\rho)\cdot f=\omega\rho$. Relabelling, or by a symmetric argument, we see that if $(\omega\rho)\cdot g=(\omega)\rho$, then $\omega\cdot g=\omega$. Hence (vi) also follows and $\B$ defines a non-trivial block system for $\Gamma$.
\end{proof}

\JordanWielandt

\begin{proof} 
Let $\Gamma_i$ be the restriction of $\Gamma$ to $\Omega_i$, so that $\Gamma$ is sub-direct in $\Gamma_1 \times \cdots \times \Gamma_k$ as in Proposition~\ref{actions and subdirect}. We let $\pi_i: \Gamma \twoheadrightarrow \Gamma_i$ denote the projection maps. 
		
	Note that our hypothesis~\ref{three} implies that each $\Gamma_i$ acts primitively on $\Omega_i$. Moreover, hypothesis~\ref{two} implies that each $\Gamma_i$ contains a non-trivial finitary permutation. Hence by Theorem~\ref{cameron}, $\Alt(\Omega_i) \leq \Gamma_i$ for all $i$. Note that $\Gamma \cap \Gamma_i$ is a normal subgroup of $\Gamma_i$ by Proposition~\ref{normal in subdirect}. If $\Gamma \cap \Gamma_i \neq1$, then it must contain $\Alt(\Omega_i)$ by Lemma~\ref{normal subs of perm groups}. Therefore, in order to prove the Theorem, it is sufficient to prove that $\Gamma \cap \Gamma_i \neq 1$ for all $i$.

	Our proof will be an induction on $k$, starting with $k=2$. As observed, if $\Gamma \cap \Gamma_1 \neq 1$ then it contains $\Alt(\Omega_1)$. But this would imply, via condition \ref{two}, that $\Gamma \cap \Gamma_2 \neq 1$. By symmetry we get that $\Gamma \cap \Gamma_1 = 1$ if and only if $\Gamma \cap \Gamma_2 = 1$. However, if $\Gamma \cap \Gamma_1 = \Gamma \cap \Gamma_2=1$ then Lemma~\ref{non trivial block system} implies that the action of $\Gamma$ admits a non-trivial block system, contradicting \ref{three}. Therefore we are done in the case $k=2$, and we have proven the base case of our induction. 
	
	Let us now suppose that $k > 2$ and that the result holds for all smaller values of $k$. We first claim that we cannot have two indices, $i \neq j$ such that $\Gamma \cap \Gamma_i = \Gamma \cap \Gamma_j = 1$. We argue by contradiction that this cannot happen. Therefore, let us assume that $\Gamma \cap \Gamma_i = \Gamma \cap \Gamma_j = 1$ to derive our desired contradiction. In that case, consider the restriction of the action of $\Gamma$ on $\Omega \setminus\Omega_i$ and note that the image of $\Gamma$ in $\Sym(\Omega \setminus \Omega_i)$ is $\Gamma \cong \Gamma/(\Gamma \cap \Gamma_i)$. It is straightforward to check that the hypotheses of the Proposition apply to the action of  $\Gamma/(\Gamma \cap \Gamma_i)$ on $\Omega \setminus\Omega_i$. In particular, $\Alt(\Omega_j) \leq \Gamma/(\Gamma \cap \Gamma_i)$.

	This means that $\pi_j( \Gamma \cap (\Gamma_i \times \Gamma_j)) \neq 1$ and therefore $ \Gamma \cap (\Gamma_i \times \Gamma_j) \neq 1$. Since we are assuming that $\Gamma \cap \Gamma_i = \Gamma \cap \Gamma_j = 1$, we may apply Lemma~\ref{non trivial block system} to derive a contradiction to \ref{three}. Hence we deduce that $\Gamma \cap \Gamma_i \neq 1$ for all but a single index. Let us suppose that, without loss of generality, $\Gamma \cap \Gamma_i \neq 1$ for $i \neq 1$. We will be done if we can show that this implies $\Gamma \cap \Gamma_1 \neq 1$. 
	
	To that end, consider the element $\gamma$ from \ref{two}. Since $\pi_1(\gamma) \neq 1$, we may find a $\sigma \in \Alt(\Omega_1)$ such that $[\pi_1 (\gamma), \sigma] \neq 1$. But as $\pi_1$ is surjective, we will have a $g \in \Gamma$ such that $1 \neq \sigma = \pi_1(g) \in \Alt(\Omega_1)$ and $[\gamma, g] \neq 1$. In particular, $\pi_i([\gamma,g]) \in \Alt(\Omega_i) $ for all $i$ and is non-trivial in $\Alt(\Omega_1)$. But now, since   $\Gamma \cap \Gamma_i \neq 1$ for $i \neq 1$ and contains $\Alt(\Omega_i)$, we get that the coset $[\gamma,g] \Alt(\Omega_2) \times \cdots \times \Alt(\Omega_k)$ contains a non-trivial element of $\Gamma \cap \Gamma_1$. 
\end{proof}

\begin{remark}
	We note that condition (ii) really is essential and one can produce examples to show that the conclusion of the Theorem fails if one drops this requirement. For instance, in the notation of the proof above, it would be insufficient to assume that each $\Gamma_i$ contains an alternating group. 
\end{remark}

For our purposes we need the following

\begin{corollary}
	\label{prodalt for houghton}
	Let $n\in \{3, 4, \ldots\}$ and  $\Gamma \leq H_n$ be a subgroup of Houghton's group of full Hirsch length, with no finite orbits and whose action on the ray system is strongly orbit primitive. Then $\Gamma$ has finitely many infinite orbits, $\Omega_1, \ldots, \Omega_k$ for some $k$. Moreover, $\Alt(\Omega_1) \times \cdots \times \Alt(\Omega_k)$ is a subgroup of $\Gamma$.  
\end{corollary}
\begin{proof}
	We simply use the previous theorem, Theorem~\ref{prodalt}. We know that $\Gamma$ has only finitely many orbits, by Lemma~\ref{4.4}. We also know that there is an element of $\gamfin$ whose support meets every (infinite) orbit by Lemma~\ref{lem:finitaryonallorbits}.  Hence the hypotheses of the Theorem are satisfied and we are done. 
\end{proof}

Recall, for any $G\le H_n$ and $X\subseteq \Ra_n$, that:
\begin{itemize}
\item $\gfin=G\cap \FSym(\Ra_n)$;
\item $\stab_G\{X\}$ denotes the setwise stabilizer;
\item $\stab_G(X)$ denotes the pointwise stabilizer; and
\item $W(G)$ is the quotient of $G\cap\stab_G\{B\}$ by $G\cap\stab_G(B)$, i.e.\
the group of automorphisms of $B$ induced by $G$.
\end{itemize}

In the case where $G$ admits a block system, $\mathcal{B} = \{ B_1, \ldots, B_k\}$, and $K$ is any subgroup of $G$ we set 
$$W_r(K)=\frac{K\cap\stab\{B_r\}}{K\cap\stab(B_r)}.$$

Given $K\le L\le G$, we have that there is a natural inclusion $W_r(K)\subseteq W_r(L)$, for any $r$. The remainder of this section is devoted to proving the following.

\begin{proposition} \label{6.1}
	Let $n\in \{3, 4, \ldots\}$ and $G\le H_n$ be a level subgroup, where every orbit is infinite. Let $\mathcal{B} = \{ B_1, \ldots, B_k\}$ be a non-trivial proper block system for $G$ and let $N$ be the kernel of the action of $G$ on $\Ra/\sim$ where $\sim$ is the congruence determined by $\mathcal{B}$. Then:
	\begin{enumerate}
		
		\item
		$\gfin$ acts transitively on each $G$-orbit. In particular, for any block $B_i$ and any $g \in G$, there exists an $x \in \gfin$ such that $B_ig=B_ix$;
		\item
		there is a fixed constant $e\in \N$, depending on $G$, such that any proper block system $\{B_1, \ldots, B_k\}$ has $|B_i|\le e$ for $i=1, \ldots, k$. In particular any proper block system is finite and $G$ must admit a maximal proper finite block system;
		\item
		$N\subseteq \gfin$; and
		\item assuming that $\mathcal{B}$ is a maximal finite block system, we have, for each $r$, that $W_r(N) = W_r(\gfin) = W_r(G)$ under the natural inclusions given above.
	\end{enumerate}
\end{proposition}
\begin{proof}
	To prove (i), let $\pi$ denote the abelianisation map: $\pi(G)= G \hfin/\hfin \cong G /\gfin$. Let $p$ be a point in the ray system, and let $G_p$ be the stabiliser of that point in $G$. We claim that $\pi(G_p) = \pi(G)$. 
	
	To see this, first use that $G$ is level to find a finite set of elements, $f_1, \ldots, f_M$ with the following properties: 
	\begin{itemize}
		\item $\langle \pi(f_j) \rangle = \pi(G)$, and 
		\item Each $f_j$ has zero translation part on some ray. 
	\end{itemize}
	
	Next note that since no orbit of $G$ is finite, by hypothesis, Lemma~\ref{inforbs} says that the $G$-orbit of $p$ meets every ray in an infinite set. In particular, the $G$-orbit of $p$ meets the fixed set of each $f_j$. Therefore, we may find elements $g_j \in G$ such that $f_j$ fixes $p {g_j}^{-1}$. Hence $f_j^{{g_j}}$ fixes $p$ for every $j$. Moreover, it is clear (since the image is abelian) that
	$$
	\langle \pi(f_j) \rangle = \langle \pi(f_j^{{g_j}}) \rangle. 
	$$  
	
	Hence we have shown that $\pi(G_p) = \pi(G)$. But this means that $G_p \gfin = G$ and hence that $\gfin$ acts transitively on $pG$ and hence every orbit of $G$. 
	
	For (ii), note that each block is finite by Lemma \ref{4.3}, since an infinite block would admit a $\gfin$ action and hence contain a $G$ orbit by part (i), contradicting the definition of a \textit{proper} block system. The existence of the constant $e$ follows from Lemma \ref{4.6}.
	
	Part (iii) follows from Lemma~\ref{3.1}, since any element of $N$ can only act with finite orbits.

	All that remains is to show (iv). Assume that $\mathcal{B}$ is maximal amongst finite block systems. We will first show that $W_r(\gfin) = W_r(G)$.
	
	We can fix a finite set $F\subset G$ of group elements so that for all $g\in G$ and all $i\ne j$, there exists $f\in F$ and a natural number $q$ such that such that the translation components of $g$ and $f^q$ agree on the $j$th ray and so that $f$ has translation component $0$ on the $i$th ray.
	Replacing $B:=B_r$ by a translate if necessary, we may assume that $B$ is entirely contained in the first ray $R_1$, and that it is in the region where all members of $F$ act in accordance with their translation component.

	Let $g$ be an element of $\stab\{B\}$. Choose $f$ so that $f$ and $g$ have the same translation component on the second ray and so that $f$ fixes $B$ pointwise. (A positive power of a suitable element of $F$ will do this.) Choose $x$ so that $Bx$ lies in the part of the second ray where both $g^{-1}$ and $f^{-1}$ act by their translation components. Hence $g^{-1}$ and $f^{-1}$ act in the same way on $Bx$, and since $g$ acts on $B$ we get that $bgxg^{-1}= bgfxf^{-1}$ for all $b \in B$. Therefore, $gxg^{-1}x^{-1}$ and $gfxf^{-1}x^{-1}$ both act in the same way on $B$; in particular, 
	$$
	b(gx g^{-1}x^{-1})( x f x^{-1}f^{-1}) = bg, \text{ for all } b \in B.
	$$
	
	 It follows that $g$ acts on $B$ in the same way as some element of the commutator subgroup, and so in particular there is a finitary permutation acting the same way as $g$, proving that $W_r(\gfin) = W_r(G)$. As $r$ was arbitrary, we have proved this for each $r$.

	Now, by part (ii), $G/N$ acts with no non-trivial proper block system on $\Ra/\sim$, and hence the result follows as in the work of Neumann, \cite[\S5, Lemma 5.4]{MR0401928}, (where it is proved that $W(N) = W(\gfin)$). In particular, we are invoking Proposition~\ref{ordering blocks}, to say that $\Ra/\sim$ is order isomorphic to $\Ra$ and that $\Gamma:=G/N$ is a Houghton group acting faithfully on the ray system $\Ra/\sim$. Furthermore, it is clear that $G/\gfin$ and $\Gamma/\gamfin$ are isomorphic, so that $\Gamma$ has full Hirsch length. Therefore Corollary~\ref{prodalt for houghton} applies to the action of $\Gamma$ on $\Ra/\sim$.  In particular, the induced $G$ action on $\Ra/\sim$, when restricted to an orbit, is $k$-transitive for every $k$. 
	
	Explicitly, we start with a $g \in \gfin$ which preserves a block, $B$. We want to show that there exists an $h \in N$ such that $g$ and $h$ induce the same permutation on $B$. 
	
	Since the support of $g$ is finite we may find blocks, $C_1, \ldots, C_r$ such that the support of $g$ is contained in the union, $B \cup  \bigcup_{i=1}^r C_i$. Let $D$ be some other block distinct from these and in the same $G$-orbit as $B$. Since $G$ acts on $\Ra/\sim$ with no non-trivial proper block system, Corollary~\ref{prodalt for houghton} implies that $\gfin$ acts highly transitively on $\Ra/\sim$ in the following sense; there exists some $x \in \gfin$ such that $C_i x = C_i$, for $1 \leq i \leq r$ and $x$ transposes $B$ and $D$. (It may also permute further blocks). 
	
	Now put $h=x^{-1} g^{-1} xg$. It is then readily checked that $h \in N$ and that $h$ induces the same permutation on $B$ as $g$ does. 
\end{proof}

We are have now assembled all the results needed to prove the following.

\structure

\begin{proof}
	Let $n \ge 3$ and $G \leq H_n$ be a subgroup of full Hirsch length, $h(G)=n-1$. By Proposition~\ref{level and inforbs}, $G$ is commensurable to a level subgroup of $H_n$ with no finite orbits. Therefore we may assume that $G$ is level, has no finite orbits and acts with finitely many orbits, by Lemma~\ref{4.4}.

	By Proposition~\ref{6.1}, $G$ admits a maximal block system, inducing a $G$-congruence, $\sim$, where each equivalence class is finite. Therefore, we get a homomorphism $\rho: G \to \Sym( \Ra / \sim)$. But by Proposition~\ref{ordering blocks}, $\Ra/ \sim$ is order isomorphic to the ray system and the image of this homomorphism is a subgroup of the Houghton group of this ray system. Thus by Theorem~\ref{kk} and Lemma~\ref{restricted}, we get an injective homomorphism $\kappa$ from $G$ to a restricted multi-wreath product, $\mathcal{W} \wr \Gamma$, where $\Gamma$ is a subgroup of the Houghton group on $\Ra/ \sim$. That is, $\Gamma \leq H_n$. 
	
	We next claim that $\Gamma$ is a subgroup of full Hirsch length. Observe that if $g \in \gfin$, then $\rho(g)$ has finite order and so must belong to $\Gamma_{\textit{fin}}$. Conversely, if $\rho(g)$ has finite order, then some power of $g$ lies in the kernel of $\rho$ and hence preserves every (finite) block. In particular, $t_i(g)=0$ for all $i$ and so $g \in \gfin$. Hence $\rho^{-1} (\Gamma_{\textit{fin}}) = \gfin$.  This implies that $G/\gfin$ is isomorphic to $\Gamma/\Gamma_{\textit{fin}}$ by the Third Isomorphism Theorem and hence $\Gamma$ is a subgroup of full Hirsch length in $H_n$. Moreover, by construction, $\Gamma$ has no finite orbits and, in fact, acts on the ray system $\Ra/ \sim$ with no proper, non-trivial block system. (If $\Gamma$ were to admit a proper, non-trivial block system, then this could be pulled back to $\Ra$, contradicting the maximality of $\sim$). Therefore, Theorem~\ref{thm-kappaG} will finish the result as long as we satisfy the hypotheses there. 
	
	Hypothesis (i) of Theorem~\ref{thm-kappaG} is proven in Proposition~\ref{6.1} (ii). Hypothesis (ii) of Theorem~\ref{thm-kappaG} is Proposition~\ref{6.1} (iv) and hypothesis (iii) of Theorem~\ref{thm-kappaG} is given by Corollary~\ref{prodalt for houghton}. 
	  
\end{proof}

We note the following special case.

\begin{corollary} \label{level and wreath}
	Let $n\in \{3, 4, \ldots\}$ and let $G \leq H_n$ be a transitive subgroup of full Hirsch length which is also level. Then $G$ is a commensurable to a restricted wreath product, $A \wr_{\Ra_n} \Gamma$, where $A$ is a finite group and $\Gamma$ is a finite index subgroup of $H_n$.  
\end{corollary}
\begin{proof}
The proof is exactly the same as above, but the fact that $G$ is already level means that we do not have to take a finite index subgroup at the first stage and so the multi-wreath product above is just a wreath product due to the transitivity of $G$. 
\end{proof}

\begin{remark}
	This last corollary allows a different structure theorem for subgroups of full Hirsch length. It allows one to show that any subgroup of full Hirsch length is commensurable to a sub-direct product of restricted wreath products (rather than a multi-wreath product). However, being commensurable to a multi-wreath product is stronger than being commensurable to a sub-direct product of wreath products. In particular, it is unclear how one would ascertain finiteness properties for sub-direct products of wreath products, as sub-direct products are in general fairly wild. Nevertheless, we do use this sub-direct structure when proving finite generation, Theorem~\ref{finite gen}. 
\end{remark}

\section{Subgroups of \texorpdfstring{$H_n$ } \   of full Hirsch length are finitely generated}\label{sec:n=2}

We will argue that subgroups of full Hirsch length in $H_n$ are finitely generated by proving this for the transitive case and then deducing it for the general case via sub-direct products. Note that the case $n=2$ is slightly different to the rest as there are subgroups of full Hirsch length of $H_2$ which do not contain the alternating group. 

\subsection*{The ascending chain condition on normal subgroups}

We will now deal with finite generation and, along the way, prove results about the ascending chain condition on normal subgroups for certain wreath products and subgroups of Houghton's group of full Hirsch length. 

\begin{definition}
	\label{maxn}
	We say that a group $G$ satisfies \textbf{max-n} if it satisfies the ascending chain condition on normal subgroups. 
	
	That is, for any chain of normal subgroups, $N_0 \leq N_1 \leq \ldots \leq N_k \leq \ldots$, there exists a natural number $s$ such that $N_s=N_{s+t}$ for all $t \geq 0$. Or, to put it another way, that any sequence of ascending normal subgroups $(N_i)_{i\in\N}$ of $G$ is eventually constant.
\end{definition}

The following lemma is standard.

\begin{lemma}
	Let $G$ be a group. Then the following are equivalent:
	\begin{enumerate}[(i)]
		\item $G$ satisfies \textbf{max-n}.
		\item Every non-empty family of normal subgroups of $G$ admits at least one maximal element under inclusion.
		\item Every normal subgroup $N$ of $G$ is finitely normally generated. That is, there exists a finite subset $F$ of $G$ such that $N$ is the normal closure of $F$. (Equivalently, $N$ is the smallest normal subgroup of $G$ containing $F$).
	\end{enumerate}
\end{lemma}

\begin{remark}
	In the subsequent discussion we will mainly use condition (iii) above when verifying \textbf{max-n}.
\end{remark}

Our main tool to proving finite generation will be the following. 

\begin{proposition}
	\label{finite gen for subdirect}
	Let $G$ be a sub-direct product of $\Gamma_1 \times \cdots \times \Gamma_k$. Suppose that:
	\begin{itemize}
		\item Each $\Gamma_i$ is finitely generated; and
		\item Each $\Gamma_i$ satisfies \textbf{max-n}. 
	\end{itemize}
Then $G$ is finitely generated and satisfies \textbf{max-n}.

\end{proposition}
\begin{proof}
We show that each desired property of $G$ is satisfied by using induction on $k$. The base case of $k=1$ is immediate for both properties.
	
	We first show that $G$ has \textbf{max-n}. Consider a sequence $(N_i)_{i\in\N}$ of ascending normal subgroups of $G$. Proposition~\ref{normal in subdirect} states, for each $i\in \N$, that $N_i \cap \Gamma_k\unlhd\Gamma_k$. Therefore, since $\Gamma_k$ satisfies \textbf{max-n}, there is an $s\in \N$ such that the subsequence $(N_{i}\cap \Gamma_k)_{i\ge s}$ is constant. 
	
	Observe, for any $i\in \N$, that $N_i \Gamma_k/\Gamma_k$ is a normal subgroup of $G \Gamma_k/\Gamma_k$. Also, $G \Gamma_k/\Gamma_k$ is sub-direct in $\Gamma_1 \times \cdots \times \Gamma_{k-1}$. Hence, by the inductive hypothesis, $G \Gamma_k/\Gamma_k$ satisfies \textbf{max-n}. Therefore, for some $t\in \N$, we have that $(N_{i+t} \Gamma_k/\Gamma_k)_{i\in\N}$ is constant. Let $d:=\max\{s, t\}$. Now, given any $m \in N_{i+d}$, there exists an $m' \in N_d$ such that $m \Gamma_k = m'\Gamma_k$ and therefore $m^{-1} m' \in N_{i+d} \cap  \Gamma_k = N_d \cap \Gamma_k$. Hence $m \in N_d$ and $G$ satisfies \textbf{max-n}.  
	
	Next we show that $G$ is finitely generated. Suppose that $k >1$ and that every group $H$ that is sub-direct in $\Gamma_1 \times \ldots  \times \Gamma_{k-1}$ is finitely generated. Let $N = G \cap \Gamma_k$. Clearly $N\unlhd G$. By Proposition~\ref{normal in subdirect}, $N\unlhd\Gamma_k$. Since $\Gamma_k$ satisfies \textbf{max-n}, there is a finite subset $F \subseteq N$ such that $N$ is generated by $\{ F^{\gamma} \ : \ \gamma \in \Gamma_k \}$. We claim that $N$ is contained in a finitely generated subgroup of $G$. 
	
	By hypothesis, $\Gamma_k$ is finitely generated. Since $G$ is sub-direct, there exists a finite subset $X\subseteq G$ such that $\pi_k(\langle X \rangle ) = \Gamma_k$. Let $\gamma \in \Gamma_k$. Then there exists some $g \in \langle X \rangle$ such that $\gamma g^{-1} \in \ker(\pi_k) \leq \Gamma_1 \times \ldots  \times \Gamma_{k-1}$. However, $\Gamma_1 \times \cdots \times \Gamma_{k-1}$ centralises $\Gamma_k$, and so $F^{\gamma} = F^g$. Hence $N \leq \langle F, X \rangle$ is contained in a finitely generated subgroup of $G$, establishing our claim. 
	
	We now invoke the induction hypothesis to say that $G/N$ is finitely generated, as it is sub-direct in $\Gamma_1 \times \ldots  \times \Gamma_{k-1}$. This implies that there exists a finite subset, $Y$ of $G$ such that for any $g \in G$, there exists a $w \in \langle Y \rangle$ such that $g w^{-1} \in N$. (That is, simply take a finite generating set for $G/N$ and pull it back to $G$). Hence $G=\langle F, X, Y\rangle$. 
\end{proof}

\begin{proposition}
	\label{transitive infinite n=2}
	Let $G$ be a level subgroup of $H_2$ which acts transitively on $\Ra_2$. Then $G$ embeds into a permutational restricted wreath product, $A \wr_{\Ra_2} \Gamma$, where $\Gamma$ is either a finite index subgroup of $H_2$ or $\Gamma \cong \Z$ acting on $\Ra_2$ by translations.

	Moreover, if we let $\pi: A \wr_{\Ra_2} \Gamma \to \Gamma$ denote the natural projection map, we have that $\pi(G) = \Gamma$. 
\end{proposition}

\begin{proof}
	We will use the results Theorem~\ref{kk} and Lemma~\ref{restricted}. 
	
	The case where $G_p\gfin = G$ for some (and hence any) $p \in \Ra_2$ implies that $\gfin$ acts transitively on $\Ra_2$. Just as in Proposition~\ref{6.1} (ii), this implies that there is a maximal finite block and $k=1$ since we are assuming the action is transitive. 
	
	We  can then use the block to produce a $G$-congruence $\sim$ and,  as in Proposition~\ref{ordering blocks}, 	$\Ra_2/\sim$ is order isomorphic to $\Ra_2$ and the image of $G$ in $\Sym(\Ra_2/\sim)$ is a subgroup of $H_2$. 
	
	We will also get Proposition~\ref{6.1} (iii), that the kernel of the map to $\Sym(\Ra_2/\sim)$ is a subgroup of $\gfin$, using the same argument. However, Proposition~\ref{6.1} (iv) can fail in general, since that argument uses the fact that we have a third ray.

	Nevertheless $\gfin$ acts transitively and $G$ acts primitively on $\Ra_2/\sim$ and hence the image of $G$ in $\Sym(\Ra_2/\sim)$ must contain the alternating group, by Theorem~\ref{cameron}. Since $G$ contains an element of infinite order which cannot lie in the kernel of the map to $\Sym(\Ra_2/\sim)$, this implies that the image of $G$ in $\Sym(\Ra_2/\sim)$  is actually a finite index subgroup of $H_2$. Hence by Theorem~\ref{kk} and Lemma~\ref{restricted}, we can embed $G$ into a wreath product $A \wr_{\Ra_2} \Gamma$, where $\Gamma$ is a finite index subgroup of $H_2$ and $A$ is a finite group. 
	
	\medspace
	
	The other situation is where $G_p \gfin = \gfin$, which implies that all point stabilisers are contained in $\gfin$. In this case, every $\gfin$ orbit must be finite (since otherwise some point stabiliser contains an element of infinite order). In this case we take our blocks to be exactly the $\gfin$ orbits. Since the $G$-action is transitive, there is only one orbit of blocks. As before, we can define a $G$-congruence, $\sim$ and a map from $G$ to $\Sym(\Ra_2/\sim)$ whose image lands in $H_2$.

	But now the kernel of the map to $\Sym(\Ra_2/\sim)$ is the whole of $\gfin$. This implies that the image of $G$ in $\Sym(\Ra_2/\sim)$ is both infinite cyclic and transitive. Hence, by Theorem~\ref{kk} and Proposition~\ref{ordering blocks}, $G$ embeds into a wreath product $A \wr_{\Ra_2} \Z$. Up to conjugation by an almost order preserving map (an element of $\FSym$, in fact) we can take the action of $\Z$ on $\Ra_2/\sim$ to be that of $\Z$ acting on $\Ra_2$ by translations. 
	
	Although we will not use it, observe that the image of $G$ in $A \wr_{\Ra_2} \Z$ is finite index, whereas the image of $G$ in $A \wr_{\Ra_2} \Gamma$ need not be. 
\end{proof}

\begin{proposition}
	\label{normal subgroups of wreaths in H_2}
	Consider the permutational restricted wreath product, $A \wr_{\Ra_n} \Gamma$, where $\Gamma$ is either a finite index subgroup of $H_n$ or $n=2$ and $\Gamma \cong \Z$ acting on $\Ra_2$ by translations and $A$ is some finite group.

	Let $\pi: A \wr_{\Ra_n} \Gamma \to \Gamma$ denote the natural projection map and let $G$ be a subgroup of $A \wr_{\Ra_n} \Gamma$ such that $\pi(G) = \Gamma$. Then 
	\begin{itemize}
		\item $G$ is finitely generated, and
		\item $G$ has \textbf{max-n}. 
	\end{itemize}
\end{proposition}
\begin{proof} 
	
We will start the proof by showing that any normal subgroup $N$ of $G$ which is contained in $G \cap B $ is finitely normally generated in $G$.

	Let $B=\ker \pi$, the base of the wreath product. Then $B = \bigoplus_{\Ra_n} A$. Since this is a direct sum, every element of $B$ has a support which is a finite subset of $\Ra_n$. We also have the projection maps, $\rho_x: B \to A$ for every $x \in \Ra_n$.

	We deal with the case where $\Gamma$ is a finite index subgroup of $H_n$.

To do this, first pick a basepoint $* \in \Ra_n$ and consider a finite subset, $F_0$ of $N$ such that $\rho_*(F_0) = \rho_*(N)$. (This is clearly possible, as $A$ is finite). 
The crucial property we will use is that $\Gamma$ contains $\Alt(\Ra_n)$ and hence acts $k$-transitively on $\Ra_n$ for any $k \in \N$.

We let $S = \bigcup_{f \in F_0} \supp(f)$ be the finite subset of $\Ra_n$ consisting of the union of all the supports of elements of $F_0$. Then let $F$ denote the finite subset of $N$ whose supports lie in $S$. 

We note that since $\Gamma$ acts transitively on $\Ra_n$ and $\pi(G) = \Gamma$, we get that for any $x \in \Ra_n$ and any $n \in N$, there exists a $g \in G$ and an $ f \in F$ such that $\rho_x(f^g) = \rho_x(n)$. In fact, we could have taken the $f \in F_0$. However, do be aware that in general, $\rho_*(f) \neq \rho_x(f^g)$, even when $xg=*$, although they are conjugate in $A$ in that case. However, $|\rho_x(N)|$ is constant as we vary $x$, since $N$ is normal and $G$ acts transitively on $\Ra_n$. Therefore the elements $\rho_x(f^g)$ exhaust $\rho_x(N)$.  

We claim that $F$ normally generates $N$. We argue by contradiction. If this is not the case, choose some $n \in N$ outside of the normal closure of $F$ whose support is minimal. There are two cases to deal with. If the support of $n$ is not larger than $|S|$ then, since $\Gamma$ acts in a highly transitive way, there exists a $g \in G$ such that $n^g$ has support contained in $S$. Hence $n^g \in F$ or, equivalently, $n \in F^{g^{-1}}$, so that $n$ lies in the normal closure of $F$ resulting in a contradiction. Otherwise, the support of $n$ has at least $|S|$ elements. Choose some $x \in \Ra_n$ which lies in the support of $n$.

Now, since the action of $\Gamma$ is highly transitive, we can find a $g \in G$ and a $f \in F$ satisfying the following: 
\begin{itemize}
	\item $\rho_x(f^g) = \rho_x(n)$
	\item The support of $f^g$ is contained in the support of $n$. 
\end{itemize}
But now, as $n$ does not lie in the normal closure of $F$, neither does ${(f^g)}^{-1} n$. But this element has a smaller support by construction. So the resulting contradiction shows that $N$ is finitely normally generated. Thus we have proved that any normal subgroup of $B$ is finitely normally generated in the case where $\Gamma$ has finite index in $H_n$. 

\medspace

We next prove the same property when $\Gamma=\Z$ and $n=2$. The proof is similar but we need to modify some details as follows. The set $F_0$ is defined in the same way. Now, since $\Ra_2 = \{ 1,2\} \times \N$, we will identify $\Ra_2$ with $\Z$ (this is not an order preserving identification, but it doesn't matter for what follows). Then the action of $\Z$ on $\Ra_2$ is simply the regular action of $\Z$ on itself. Given any finite subset $S$ of $\Z$, we then assign an interval, $I(S):=[\inf(S),\sup(S)]_{\Z}$ and the length of $S$, $l(S)$ to be $|\sup(S) - \inf(S)|$. 

Note that for any $b \in B$, its support has a length and this length is a conjugacy invariant. 

Now let $L$ be the maximum of the lengths of elements of $F_0$ and extend $F_0$ to a larger finite set, $F$, such that any $n \in N$ of length $L$ is conjugate to an element of $F$. As before, we claim that $F$ normally generates $N$ in $G$. This follows as before. We argue by contradiction, choosing an $n \in N$ of minimal length which does not lie in the normal closure of $F$. By construction, as before, the length of $n$ must be greater than $L$ (as otherwise it would be equal to a conjugate of an element of $F$). Then, as before, we can find an $x$ in the support of $n$, a $g \in G$ and a $f \in F$ such that: 
\begin{itemize}
	\item $\rho_x(f^g) = \rho_x(n)$
	\item The interval of the support of $f^g$ is contained in the interval of the support of $n$. 
\end{itemize}
It is then clear that the length of the support of ${(f^g)}^{-1} n$ is smaller than the length of $n$; this contradiction proves that $F$ normally generates $N$.

\medspace

Thus we have shown in either case that any normal subgroup of $G \cap B$ is finitely normally generated in $G$. We now use this to get the remaining results.

\medspace

We next prove that $G$ is finitely generated. Since $G \cap B$ is a normal subgroup of $G$, it is normally generated by a finite set, $F$. Moreover by possibly enlarging $F$, we assert that there is a finite subset $S$ of $\Ra_n$ so that:
\begin{itemize}
	\item Any $f \in F$ has support contained in $S$ and, 
	\item Any element whose support is contained in $S$ is in $F$. 
\end{itemize}
That is, $F$ consists of exactly those (finitely many) elements whose support is contained in a fixed finite set, $S$. (We are using the fact that $A$ is finite here.)

Now since $\Gamma$ is finitely generated and $\pi(G)= \Gamma$ we can find a finite subset $X$ of $G$ such that $\pi(\langle X \rangle) = \Gamma$. That is, we take pre-images of a finite generating set. We claim that $G$ is generated by $X$ and $F$. Since $F$ normally generates $G \cap B$ and $\pi(\langle X \rangle) = \Gamma$, it is sufficient to show that $F^g \subseteq \langle X, F \rangle$ for all $g \in G$. Given such a $g$, there exists a $w \in \langle X \rangle$ such that $\pi(g)=\pi(w)$. But then the support of $F^{g w^{-1}}$ equals the support of $F$ (since $g w^{-1}$ acts trivially on $\Ra_n$). Hence $F^{g w^{-1}}  = F$, by the comments above (and the fact that $F$ is finite). In particular, $F^g = F^w \subseteq \langle X, F \rangle$. Hence $G$ is generated by $X$ and $F$. This shows that $G$ is finitely generated. 

\medspace

Finally, we prove that an arbitrary normal subgroup $N$ of $G$ is finitely normally generated. We have already proved that $N \cap B$ is finitely normally generated. And since every normal subgroup of $\Gamma$ is finitely normally generated, it follows that $NB/B \leq GB/B \cong \Gamma$ is also finitely normally generated, since $NB/B$ is a normal subgroup of $GB/B$. (Note when $\Gamma$ is a finite index subgroup of $H_n$, then a normal subgroup is either trivial, $\Alt, \FSym$ or finite index in $H_n$ and hence is finitely normally generated. When $\Gamma=\Z$, then every subgroup is finitely generated.)

It then follows that there is a finite subset $F_1$ of $N$ such that $F_1B$ normally generates $NB/B$ and a finite subset, $F_2$ of $N \cap B$ which normally generates $N \cap B$. Hence $N$ is finitely normally generated by $F_1 \cup F_2$. 
\end{proof}

\begin{theorem}
	\label{finite gen}
Let $n\ge2$. Then every subgroup of full Hirsch length in $H_n$ is finitely generated and has \textbf{max-n}. 
\end{theorem}
\begin{proof}
	Let $G$ be a subgroup of $H_n$ of full Hirsch length. Since both finite generation and \textbf{max-n } are commensurability invariants (see \cite[Exercise 3.1.11]{Robinson} for the latter) we can assume that $G$ is level and acts with finitely many orbits on $\Ra_n$, each of which is infinite, by Proposition~\ref{level and inforbs} and lemma~\ref{4.4}. We use these orbits to realise $G$ as a sub-direct product as in Proposition~\ref{actions and subdirect}.

	Therefore, as in Proposition~\ref{actions and subdirect}, $G$ is sub-direct in $\widehat{\Gamma_1} \times \cdots \times \widehat{\Gamma_k}$, where each $\widehat{\Gamma_i}$ is a level and transitive subgroup of $H_n$ of full Hirsch length by Lemma~\ref{restricting}. Next we invoke Proposition~\ref{transitive infinite n=2} for $n=2$ and	Corollary~\ref{level and wreath} for $n \ge 3$ to realise that each $\widehat{\Gamma_i}$ has a wreath product structure.

	Concretely, this implies that each $\widehat{\Gamma_i}$ is isomorphic to a subgroup of a permutational restricted wreath product, $A \wr_{\Ra_n} {\Gamma_i}$, where ${\Gamma_i}$ is either a primitive subgroup of $H_n$ or $n=2$ and ${\Gamma_i} \cong \Z$ acting on $\Ra_2$ by translations. (And, additionally, the natural projection map from $\widehat{\Gamma_i}$ to ${\Gamma_i}$ is surjective). In the former case,  ${\Gamma_i}$ is a finite index subgroup of $H_n$ by Corollary~\ref{cameron2}. (Actually, for $n \geq 3$, each $\widehat{\Gamma_i}$ is a finite index subgroup of the wreath product). 
	
	But by Proposition~\ref{normal subgroups of wreaths in H_2}, $\widehat{\Gamma_i}$ is finitely generated and has \textbf{max-n}. Hence $G$ is also finitely generated by Proposition~\ref{finite gen for subdirect} and has \textbf{max-n}.
\end{proof}

\section{Proving the Main Theorem}\label{sec:thmAproof}
From our Theorem \ref{thm:structure}, we know that a subgroup of full Hirsch length is abstractly commensurable to a multi-wreath product of a particular form. Our approach is to use the BNS Invariants to show that the head that occurs within such a wreath product is of type $\FFF_{n-1}$. We will then use \cite[Theorem A]{KM1} to deduce that the wreath product, and hence also the original group, is of type $\FFF_{n-1}$. With this approach in mind, we begin by recalling the necessary results for BNS Invariants. We direct a reader new to these ideas to \cite{BNSdefinitions}.

\begin{definition}
	Let $G$ be a finitely generated group and $m:=\dim_{\Q} G/[G,G] \otimes_{\Z} \Q$. Then $\Hom(G, \R) \cong \R^{m}$, a real vector space of dimension $m$. The \emph{character sphere of} $G$ is then
	$$S(G):=\Hom(G, \R) - \{ 0 \}/\sim,$$
	where $\sim$ is homothety (multiplication by a positive real). Hence, $S(G)$ is homeomorphic to a sphere of dimension $m-1$.  
\end{definition}

Let $H_n$ be the Houghton group on $n$ rays. We now summarise the key results from \cite{BNSRHoughton} on the BNS Invariants for $H_n$. This has abelianisation $\Z^{n-1}$. Hence $\Hom(H_n, \R) \cong \R^{n-1}$, meaning that the character sphere of $H_n$ is $S^{n-2}$. 

More concretely, we have $n$ translation maps, $t_1, \ldots, t_n$, which register the eventual translation on each respective ray. Then each $t_i$ is a $\Z$-valued function in $\Hom(H_n, \R)$. Moreover, $\sum_{i=1}^n t_i$ is the zero map. In fact, $t_1, \ldots, t_n$ span $\Hom(H_n, \R)$.

We can then form the $n-1$-simplex,
$$\Delta(t_1, \ldots, t_n) = \{ \sum_{i=1}^n a_i t_i \ : \ a_1,\ldots, a_n\ge 0, \sum_{i=1}^n a_i = 1\}$$
inside $\Hom(H_n, \R)$. The boundary of this $n-1$-simplex is homeomorphic to $S^{n-2}$, 
$$
S^{n-2} \cong \partial  \Delta(t_1, \ldots, t_n) = \{ \sum_{i=1}^n a_i t_i \ : \  a_1,\ldots, a_n\ge 0, \sum_{i=1}^n a_i = 1, \prod_{i=1}^n a_i = 0\}.
$$ 

From the above facts, that the $t_i$ span $\Hom(H_n, \R)$ and that $\sum_{i=1}^n t_i =0$, it is easy to see that $\partial  \Delta(t_1, \ldots, t_n)$ can be identified with the character sphere of $H_n$. More precisely, the natural quotient map from $\Hom(H_n, \R) - \{0 \}$ to the character sphere is a homeomorphism when restricted to $\partial  \Delta(t_1, \ldots, t_n)$. 

Note that this gives a triangulation of the character sphere of $H_n$. The BNS invariants for $H_n$ can then be described as follows: the complement of the $i^{th}$ homotopical BNS invariant is the $i-1$ skeleton of the character sphere with respect to the triangulation above. For instance, the first homotopical BNS invariant is the complement of the $0$-skeleton. Hence, by Theorem~\ref{thm:BNS} below, the kernel of each $\pm t_i$ fails to be finitely generated and these are the only points in the character sphere whose kernels fail to be finitely generated.

We next consider direct products, due to the head of our wreath product being subdirect in $\Gamma_1\times\cdots\times\Gamma_k$. As seen in \cite{BNSdirectprods}, the character sphere of $k$-fold direct product can be readily identified with the $k$-fold join of the character spheres of the direct summands. Moreover, we have the following inclusion (Meinert's Inequality):

$$
\Sigma^m(G^k)^c \subseteq \bigcup_{a_1+\cdots+a_k=m} \Sigma^{a_1}(G)^c* \ldots * \Sigma^{a_k}(G)^c,
$$ 
where $\Sigma^m$ denotes the $m^{th}$ homotopical BNS invariant. Note that this formula is only valid when $G$ (and hence $G^k$) is of type $\FFF_m$. In particular, when $G=H_n$, the character sphere of $G^k$ has an induced triangulation as the join of triangulated $S^{n-2}$ spheres. The character sphere of $(H_n)^k$ is a $k(n-1) -1$ sphere. 

Taking the formula when $m=n-1$ (which is the maximum value for which it is valid) and using the results above about the BNS invariants for Houghton groups, we get that $\Sigma^{n-1}({H_n}^k)^c$ is contained in the $n-2$-skeleton of the triangulation above of the  $k(n-1) -1$ sphere.

\begin{remark} The values above rely on the fact that the join of simplicial complexes of dimensions $n$ and $m$ results in a simplicial complex of dimension $n+m+1$. Hence $ \Sigma^{a_1}(G)^c* \ldots * \Sigma^{a_k}(G)^c$ has dimension $\sum_{i=1}^k \dim(\Sigma^{a_i}(G)^c) + (k-1)$. When $G=H_n$, we have that $\dim(\Sigma^{a_i}(G)^c) = a_i-1$ and if we set $m$ as $n-1$ we also get $a_1+\ldots +a_k = n-1$ and whence the result.
\end{remark}

We shall use the following result to determine if a subgroup of $G$ containing the derived subgroup of $G$ has some finiteness property.

\begin{theorem}[Theorem B, \cite{BNSdefinitions}]
\label{thm:BNS}	 Let $G$ be a finitely generated group of type $\FFF_k$ and $H$ a subgroup of $G$ which contains the derived subgroup of $G$. Then, 
\begin{enumerate}
	\item $S(G,H) = \{ [\chi] \in S(G) \ : \ \chi(H) = 0 \}$, is the great sphere of $H$ in $G$.
	\item For any $m \leq k$, $H$ is of type $\FFF_m$ if and only if $S(G,H)$ is contained in the $k^{th}$ sigma invariant of $G$; that is,
	$$
	H \text{ of type } \FFF_m \iff S(G,H) \subseteq \Sigma^m(G).
	$$
\end{enumerate}
\end{theorem}
The following well-known lemma will be helpful for the theorem that follows.
\begin{lemma}\label{lem:indexresult} Let $G\le H$ and $N\unlhd H$. Then $G$ has finite index in $GN$ if and only if $G\cap N$ has finite index in $N$.
\end{lemma}
\begin{proof} Let $X$ be the set of (right) cosets of $G$ in $GN$. Then $|X|$ is the index of $G$ in $GN$. However, $N$ also acts transitively on $X$ by right multiplication, and the stabiliser in $N$ of $G$ is $G \cap N$. Hence also $|X|$ is the index of $G \cap N$ in $N$.
\end{proof}
We now summarise the key properties that hold with our setup, and use these to show that a subgroup of $H_n$ of full Hirsch length is of type $\FFF_{n-1}$.
\begin{theorem}
	Let $n \geq 3$ and $G\le H_n$ be of full Hirsch length with $k$ orbits, $\Omega_1, \ldots, \Omega_k$, each of which is infinite. Suppose further that $G$ is a level subgroup and acts strongly orbit primitively on the ray system. Then: 
	\begin{enumerate}
		\item Each $\Omega_i$ is order isomorphic (as a well ordered set) to the ray system, $\Ra_n$.
		\item $G$ embeds as a subgroup of $\Gamma_1 \times \cdots \times\Gamma_k$, where $\Gamma_i$ is the full Houghton group on $\Omega_i$. 
		\item $G$ contains a subgroup commensurate to the derived subgroup of  $\Gamma_1 \times \cdots\times \Gamma_k$.
		\item For each projection map, $\pi_i: \Gamma_1 \times \ldots\times \Gamma_k \to \Gamma_i$ we have that $\pi_i(G)$ is a finite index subgroup of $\Gamma_i$.
	\end{enumerate}
	
	Moreover, $G$ is of type $\FFF_{n-1}$. 
\end{theorem}
\begin{proof}
	We first observe that $G\le \Sym(\Omega_1) \times \ldots\times \Sym(\Omega_k)$. Then Lemma \ref{restricting}, which builds on Lemma~\ref{inforbs}, provides (i) and  (ii).  For (iv), note that each $\pi_i(G)$ must act primitively on $\Omega_i$ and hence, by Corollary~\ref{cameron2}, is a finite index subgroup of $\Gamma_i$. Lemma \ref{lem:finitaryonallorbits} states the existence of an element satisfying assumption (ii) of Theorem~\ref{prodalt}, from which our statement (iii) above, follows.
	 	 
	We now wish to show that $G$ is of type $\FFF_{n-1}$. Consider the translation homomorphisms, $t_1, \ldots, t_n$ from $H_n$ (and hence also $G$) to $\Z$ based on the original ray system. We then have, for each $i$, translation maps $t_{ij}$, each of which measures the eventual translation along the $i$th ray in $\Omega_j$. These are related in the following way for all $g \in G$: 
	
	\begin{itemize}
		\item If $t_i(g)=0$, then $t_{ij}(g) = 0$ for all $j$. 
		\item If $t_i(g) >0$, then $t_{ij}(g) > 0$ for all $j$. 
		\item If $t_i(g) <0$, then $t_{ij}(g) < 0$ for all $j$.  
	\end{itemize}
	 
This is clear, but do note that the numbers themselves need not be equal. (However, there will exist a $0 < q_j \in \Q$ such that $t_{ij}(g) = q_j t_i$ for all $g \in G$.) From this observation, any element of the character sphere of $\Gamma_1 \times \cdots \times\Gamma_k$ can be represented by some 
$$
\chi = \sum_{i,j} a_{ij} t_{ij}, 
$$
where $0 \leq a_{ij} \in \R$. (We can also impose the condition that $\sum a_{ij} = 1$, but this will not be needed.)

Our condition (iii) states that the derived subgroup of $G$ has finite index in $N=\bigoplus_{i}\FSym(\Omega_i)$, which is the derived subgroup of $\Gamma_1 \times \cdots \times\Gamma_k$. Thus Lemma \ref{lem:indexresult} applies, and $G$ has finite index in $GN$. Up to commensurability, therefore, we may assume that $G$ contains $N$. We can make such a change since, for any $m\ge 2$, the finiteness conditions $\FFF_m$ and $\FP_m$ are commensurability invariants; see \cite[Corollary 9]{Alonso}. We now have that $G$ will be of type $\FFF_{n-1}$ as long as the great sphere of $G$ in $\Gamma_1 \times \cdots \times \Gamma_k$ is contained in the $n-1$ sigma invariant of $\Gamma_1 \times \cdots \times \Gamma_k$.

Thus if $G$ fails to be of type $\FFF_{n-1}$, then there exists a $\chi = \sum_{i,j} a_{ij} t_{ij}$ such that $\chi(G) = 0$ and at most  $n-1$ of the $a_{ij}$ are non-zero (while all the rest are zero, and all of them are non-negative). That is, if $G$ fails to be of type $\FFF_{n-1}$ then it is contained in the kernel of some homomorphism to $\R$ which is in the $n-2$ skeleton of the triangulation given above.

However, given such a $\chi$, there must exist a ray, some $1 \leq i_0 \leq n$ such that $a_{i_0 j} = 0 $ for all $j$. We must also have some coefficient, $a_{i_1 j_1} > 0$.

Since $G$ is level, we can find a $g \in G$ such  $t_{i_0}(g) < 0$, $t_{i_1}(g) > 0$ and $t_i(g) = 0 $ for $i \neq i_0, i_1$. But then $t_{i_1 j}(g) > 0 $ for all $j$ and $t_{ij}(g) = 0$ for $i \neq i_0, i_1$. Now, since $a_{i_0 j} = 0 $ for all $j$, we get that $\chi(g) > 0$. This is a contradiction and hence we deduce that $G$ is of type $\FFF_{n-1}$. 
\end{proof}
Our aim is now to apply \cite[Theorem A]{KM1}. This provides 3 conditions which are together sufficient for the graph wreath product to have the finiteness condition $\FFF_n$.
\begin{enumerate}
\item The head, $H$, is of type $\FFF_n$.
\item The group we are wreathing over, $A$, is of type $\FFF_n$.
\item $\Z\Delta_p$ is of type $\FP_{n-1-p}$ over $\Z H$ for $0\le p\le n-1$.
\end{enumerate}
Note that we have just shown condition (i). Condition (ii) is immediate since in our setting these groups are finite. Condition (iii) is reformulated in \cite[Lemma 1.8]{KM1} as follows.
\begin{lemma}[\cite{KM1}, Lemma 1.8] \label{graphwreathresult}
Let $H$ be a group and let $\Gamma$ be a non-empty simple $H$-graph. Let $L$ be the flag complex spanned by $\Gamma$. Let $m$ and $n$ be non-negative integers. Let $\Delta_m$ denote the set of $m$-simplices of $L$ and let $L^m$ denote the $m$-skeleton of $L$, for $m \ge 0$. Then the action of $H$ extends naturally to $L$ and the following are equivalent:
\begin{enumerate}
\item $\Z\Delta_m$ is of type $\FP_n$ as a $\Z H$-module.
\item $H$ has finitely many orbits of $(m+1)$-cliques, and the stabiliser of each such clique is of type $\FP_n$.
\end{enumerate}
\end{lemma}
We note that condition (ii) of the preceding lemma applies in our context. The following implies Theorem \ref{main}.
\begin{proposition} \label{Type Fn}  Let $n\ge 3$ and $G\le H_n$ have $h(G)=n-1$. Then $G$ is of type $\FFF_{n-1}$.
\end{proposition}
\begin{proof} Note that $G$ is abstractly commensurable to $\mathcal{W}\wr \Gamma$, where $\mathcal{W}$ and $\Gamma$ satisfy the conditions of Theorem \ref{thm:structure}. Then $\mathcal{W}\wr \Gamma$ has orbits $\Omega_1, \ldots, \Omega_k\subseteq \Ra_n$, all of which are infinite. Our aim is to show that conditions (i)-(iii) of \cite[Theorem A]{KM1} are satisfied by $\mathcal{W}\wr \Gamma$. Conditions (i) and (ii) are immediate. Form a graph $X$ with vertex set $\Ra_n$ and edge set $\{(v, w) : \{v, w\}\subset \Omega_d$ for some $d\in \{1, \ldots, k\}\}$. Since $\mathcal{W}\wr \Gamma$ contains $\Alt(\Omega_1)\times\cdots\times\Alt(\Omega_k)$, $\mathcal{W}\wr \Gamma$ has finitely many orbits of $(m+1)$-cliques and the stabilizer of each $(m+1)$-clique is of type $\FP_n$. We can therefore apply Lemma \ref{graphwreathresult} so that condition (iii) of \cite[Theorem A]{KM1} is satisfied.
\end{proof} 

\main

\begin{proof}
	Let $G$ be a subgroup of $H_n$ of full Hirsch length. Then $G$ is finitely generated, which is to say of type $\FFF_1$, by Theorem~\ref{finite gen} and also has \textbf{max-n} by the same result. Moreover, if $n \ge 3$, then $G$ is of type $\FFF_{n-1}$ by Proposition~\ref{Type Fn}. Hence for all $n \ge 2$, $G$ is of type $\FFF_{n-1}$. 
	
	For $n \ge 3$ Theorem~\ref{NoFPn} says that $G$ is not of type $\FP_{n}$, since $h(G)=h(H_n)$. Whereas the same result says that if $n=2$, either $G$ is not of type $\FP_2$ or is finite-by-$\Z$. 
\end{proof}

\end{document}